\documentclass[11pt,letterpaper,reqno]{amsart}
\usepackage{vmargin}
\usepackage{color}
\usepackage[latin1]{inputenc}
\usepackage{amsmath}
\usepackage{amsfonts}
\usepackage{amsthm}
\usepackage{epsfig}
\usepackage{graphicx}
\usepackage{amssymb}
\usepackage{esint}
\usepackage{caption}

\newcommand{\E}{\mathcal{E}}

\renewcommand{\H}{\mathcal{H}}
\renewcommand{\L}{\mathcal{L}}
\newcommand{\M}{\mathcal{M}}

\newcommand{\X}{\mathcal{X}}
\newcommand{\Y}{\mathcal{Y}}

\newcommand{\U}{\mathcal{U}}

\newcommand{\R}{\mathbb{R}}

\newcommand{\Om}{\Omega}

\renewcommand{\a}{\alpha}
\renewcommand{\b}{\beta}
\newcommand{\g}{\gamma}
\newcommand{\de}{\delta}
\newcommand{\e}{\varepsilon}
\renewcommand{\k}{\kappa}
\renewcommand{\l}{\lambda}
\newcommand{\s}{\sigma}

\newcommand{\vphi}{\varphi}

\newcommand{\Div}{{\rm div}\,}
\newcommand{\Id}{{\rm Id}\,}

\newcommand{\spt}{{\rm spt}}

\newcommand{\weak}{\rightharpoonup}
\newcommand{\weakstar}{\stackrel{*}{\rightharpoonup}}

\newcommand{\ov}{\overline}

\newcommand{\pa}{\partial}

\newcommand{\pp}{\mathbf{p}}

\newcommand{\eps}{\varepsilon}

\newcommand{\CC}{\textbf{\textup{C}}}
\newcommand{\DD}{\textbf{\textup{D}}}

\newcommand{\pstar}{{p^{\star}}}
\newcommand{\psharp}{{p^{\#}}}
\newcommand{\vv}{\mathsf{v}}
\renewcommand{\tt}{\mathsf{t}}
\renewcommand{\gg}{\mathsf{g}}
\newcommand{\ISO}{{\rm ISO}}
\newcommand{\na}{{\nabla}}

\theoremstyle{plain}
\newtheorem{theorem}{Theorem}[section]
\newtheorem{lemma}[theorem]{Lemma}

\newtheorem*{theorem*}{Theorem}
\newtheorem*{corollary*}{Corollary}

\theoremstyle{definition}

\newtheorem{remark}[theorem]{Remark}

\newtheorem*{notation*}{Notation}

\numberwithin{equation}{section}
\numberwithin{figure}{section}

\title[Rigidity theorems for best {S}obolev inequalities]{Rigidity theorems for best {S}obolev inequalities}

\author{Francesco Maggi}
\address{Department of Mathematics, The University of Texas at Austin, 2515 Speedway, Stop C1200, Austin TX 78712-1202, United States of America}
\email{maggi@math.utexas.edu}
\author{Robin Neumayer}
\address{Department of Mathematical Sciences, Carnegie Mellon University, Pittsburgh, PA 15213, United States of America}
\email{neumayer@cmu.edu}
\author{Ignacio Tomasetti}
\address{Department of Mathematics, The University of Texas at Austin, 2515 Speedway, Stop C1200, Austin TX 78712-1202, United States of America}
\email{tomasetti@math.utexas.edu}

\begin{document}

\begin{abstract}
{\rm For $n\ge 2$, $p\in(1,n)$, the ``best $p$-Sobolev inequality'' on an open set $\Om\subset\R^n$ is identified with a family $\Phi_\Om$ of variational problems with critical volume and trace constraints. When $\Om$ is bounded we prove: (i) for every $n$ and $p$, the existence of generalized minimizers that have at most one boundary concentration point, and: (ii) for $n> 2\,p$, the existence of (classical) minimizers. We then establish rigidity results for the comparison theorem ``balls have the worst best Sobolev inequalities'' by the first named author and Villani, thus giving the first affirmative answers to a question raised in \cite{maggivillaniJGA}.}
\end{abstract}


\maketitle

\tableofcontents

\section{Introduction} \subsection{Overview}\label{subsection overview} The goal of this paper is to answer some basic open questions concerning a ``doubly critical'' family $\{\Phi_\Om(T)\}_{T\ge0}$ of minimization problems on Sobolev functions, which, in a precise sense to be clarified below, can be interpreted as collectively defining the best Sobolev inequality on an open set $\Om\subset\R^n$ with $C^1$-boundary. Given an integer $n\ge 2$ and $p\in(1,n)$, these problems are defined as
\begin{equation}
  \label{phi Omega T}
  \Phi_\Om(T)=\inf\Big\{\Big(\int_\Om|\nabla u|^p\Big)^{1/p}:\int_{\Om}|u|^\pstar=1\,,\,\int_{\pa\Om}|u|^\psharp=T^\psharp\Big\}\,,
\end{equation}
and their minimizers, whenever they exist, satisfy the Euler--Lagrange equation
\begin{equation}
  \label{PDE of generic minimizer}
  \left\{
  \begin{split}
    &-\Delta_p u=\l\,u^{\pstar-1}\,,\quad\hspace{1cm}\mbox{on $\Om$}\,,
    \\
    &|\nabla u|^{p-2}\,\frac{\pa u}{\pa\nu_\Om}=\s\,u^{\psharp-1}\,,\quad\mbox{on $\pa\Om$}\,,
  \end{split}\right .
\end{equation}
for suitable Lagrange multipliers $\l,\s\in\R$.  We call $\int_{\Om}|u|^\pstar=1$ and $\int_{\pa\Om}|u|^\psharp=T^\psharp$ the ``volume'' and ``trace'' constraints of $\Phi_\Om(T)$. The critical Sobolev exponents associated to $n$ and $p$,  $\pstar$ and $\psharp$, are defined by
\[
\pstar=\frac{np}{n-p}\,,\qquad\psharp=\frac{(n-1)p}{n-p}=\frac{n-1}{n}\,\pstar\,,
\]
and their precise values guarantee the scale invariance of $\Phi_\Om$, i.e.
\[
\Phi_{x+r\,\Om}(T)=\Phi_\Om(T)\qquad\forall x\in\R^n\,,r>0\,, T\ge0\,.
\]
When $\Om$ is bounded, the $C^1$-regularity of $\pa\Om$ guarantees that every $u\in L^1_{\rm loc}(\Om)$ with $\nabla u\in L^p(\Om;\R^n)$ lies in the competition class of $\Phi_\Om(T)$ for some  $T\ge0$. In particular,
\[
{\rm Epi}(\Phi_\Om)=\big\{(T,G)\in\R^2:T\ge0\,,G\ge\Phi_\Om(T)\big\}\,,
\]
(the epigraph of $\Phi_\Om$) collects the best possible information on the range of values achievable by $\|\nabla u\|_{L^p(\Om)}$ when $\|u\|_{L^\pstar(\Om)}$ is fixed: from this peculiar viewpoint, which is reminiscent of the one adopted in the study of Blaschke--Santal\'o diagrams, ${\rm Epi}(\Phi_\Om)$ {\it is} ``the best Sobolev inequality on $\Omega$''. The following list of results, summarized in
\begin{figure}
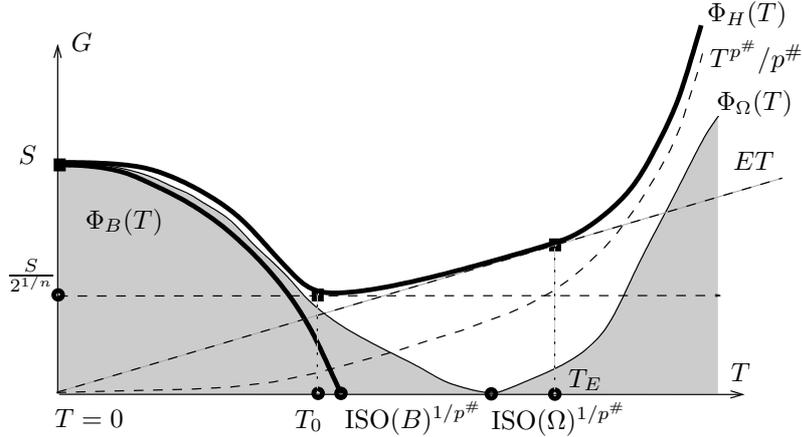
\caption{\small{The state of the art about $\Phi_\Om$. The ``exclusion zone'' for the values of $\|\nabla u\|_{L^p(\Om)}$ (under the constraint $\|u\|_{L^\pstar(\Om)}=1$) is depicted in gray. It always contains the subgraph of $\Phi_B(T)$ on $T\in[0,\ISO(B)^{1/\psharp}]$, see \eqref{mv comparison}, but is always smaller than the one of a half-space $H$, see \eqref{mn comparison}. The Sobolev inequality on $\R^n$ is equivalent to $\Phi_\Om(0)=S(n,p)$, the Euclidean isoperimetric inequality to the fact that the only zero of $\Phi_\Om$ (i.e. $T=\ISO(\Om)^{1/\psharp}$) is achieved to the right of the only zero of $\Phi_B$, and the Escobar inequality \eqref{escobar inequality} is equivalent to the linear bound $\Phi_H(T)\ge E\,T$. Both $\{(T,\Phi_B(T)):T\in[0,\ISO(B)^{1/\psharp}]\}$ and $\{(T,\Phi_H(T)):T\ge0\}$ can be implicitly parametrized by looking at the explicit families of minimizers given in \eqref{PhiB characterization}, \eqref{PhiH sobolev range}, \eqref{PhiH escobar} and \eqref{PhiH beyond E range}.}}
\label{fig summary}
\end{figure}
Figure \ref{fig summary}, aims to provide a hopefully complete state of the art on $\Phi_\Om$, and illustrates the wealth of information stored in this family of variational problems. As a disclaimer: here we are definitely {\it not} attempting to exhaustively frame the study of $\Phi_\Om$ into the incredibly vast and layered context of the theory of Sobolev-type inequalities (see e.g. \cite{mazyaBOOKonSobolev}), as that would be a long and delicate exercise, lying well beyond the scope of this introduction.

\medskip

\noindent {\bf (1) Sobolev inequality on $\R^n$:} A scaling and localization argument shows that, for every open set $\Omega$, one has
\begin{equation}
  \label{sobolev inequality on Rn}
  \Phi_\Om(0)=S(n,p):=\inf\Big\{\Big(\int_{\R^n}|\nabla u|^p\Big)^{1/p}:\int_{\R^n}|u|^\pstar=1\Big\}\,,
\end{equation}
that is, $\Phi_\Om(0)$ is the best constant in the $L^p$-Sobolev inequality on $\R^n$. Minimizers of \eqref{sobolev inequality on Rn} are exactly given by the family $\{\tau_{x_0}[U_S^{(\a)}]\}_{x_0\in\R^n\,,\a>0}$ generated by
\begin{equation}
  \label{U Sobolev}
  U_S(x)=\big(1+|x|^{p/(p-1)}\big)^{1-(n/p)}\,,\qquad x\in\R^n\,,
\end{equation}
see \cite{aubin76,talenti76,CENV2004}, we see that $\Phi_\Om(0)$ is attained if and only if $\Om=\R^n$. Here and in the following, we set
\begin{equation}
  \label{def of tau x0 and of rho alfa}
  \tau_{x_0}[v](x)=v(x-x_0)\,,\qquad v^{(\a)}(x)=\a^{-n/\pstar}\,v(x/\a)
\end{equation}
whenever $x_0\in\R^n$ and $\a>0$.

\medskip

\noindent {\bf (2) Escobar inequality:} The Escobar inequality (\cite[for $p=2$]{escobarTRACE_INQ88}, \cite[for $p\in(1,n)$]{Nazaret06}) states that if $H$ is an (open) half-space in $\R^n$ with outer unit normal $\nu_H$, then
\begin{equation}
  \label{escobar inequality}
  \Big(\int_H|\nabla u|^p\Big)^{1/p}\ge E(n,p)\,\Big(\int_{\pa H}|u|^\psharp\Big)^{1/\psharp}
  \end{equation}
with equality if and only if\footnote{The ``only if'' statement for $p \neq 2$ was left open in \cite{Nazaret06}, but was proven in \cite[Theorem 2.3]{maggineumayer}.}
 $u=\tau_{x_0}[U_E^{(\a)}]$ for some $x_0\in\R^n\setminus\ov{H}$, and where
\begin{equation}
  \label{U Escobar}
  U_E(x)={|x|^{-(n-p)/(p-1)}}\,,\qquad x\in\R^n\setminus\{0\}\,,
\end{equation}
is a multiple of the fundamental solution of the $p\,$-Laplacian. The quantity
\begin{equation}
  \label{def of TE}
  T_E=\|\tau_{x_0}[U_E^{(\a)}]\|_{L^{\psharp}(\pa H)}\,\Big/\,\|\tau_{x_0}[U_E^{(\a)}]\|_{L^{\pstar}(H)}
\end{equation}
is independent of $x_0\in\R^n\setminus\ov{H}$, and is such that
\begin{equation}
  \label{PhiH of TE}
  \Phi_H(T_E)=E(n,p)\,.
\end{equation}
The Escobar inequality \eqref{escobar inequality} can be equivalently reformulated in the ``$\Phi$-setting'' as a linear lower bound for $\Phi_H$, i.e.
\begin{equation*}
  \label{escobar inequality equivalent}
  \Phi_H(T)\ge E(n,p)\,T\,,\qquad\forall T\ge0\,.
\end{equation*}
This bound is sharp only if $T=T_E$, and is nearly optimal only if $T$ is close to $T_E$; but it largely suboptimal away from $T_E$, see Figure \ref{fig summary}.

\medskip

\noindent {\bf (3) Euclidean isoperimetry:} It is easily seen that $\Phi_\Om(T)=0$ for some $T>0$ if and only if $|\Om|<\infty$ and $T=\ISO(\Om)^{1/\psharp}$, where $\ISO(\Om)=\H^{n-1}(\pa\Om)/|\Om|^{(n-1)/n}$ stands for the isoperimetric ratio of $\Om$. Since the Euclidean isoperimetric inequality states that
\begin{equation}
  \label{euclidean isoperimetry}
  \ISO(\Om)\ge\ISO(B)\,,
\end{equation}
(with equality if and only if $\Om$ is a ball), in the $\Phi$-setting, \eqref{euclidean isoperimetry} is equivalent to saying that $\Phi_B$ has the left-most zero among all $\Phi_\Om$.

\medskip

\noindent {\bf (4) Balls have the worst best Sobolev inequalities:} In \cite[$p=2$]{carlenloss1,carlenloss2} (by symmetrization methods and conformal invariance) and in \cite[$p\in(1,n)$]{maggivillaniJGA} (via the mass transportation method pioneered in \cite{knothe,Gromov,CENV2004}) it is shown that if $B$ is a ball, then for every $T\in(0,\ISO(B)^{1/\psharp})$ there is a unique $\a>0$ such that
\begin{equation}\label{PhiB characterization}
  \Phi_B(T)=\|\nabla U_S^{(\a)}\|_{L^p(B)}\,\Big /\,\|U_S^{(\a)}\|_{L^\pstar(B)}\,,
\end{equation}
with $U_S$ defined in \eqref{U Sobolev}. Further elaborating on the proof of this partial characterization of $\Phi_B$, again in \cite{maggivillaniJGA} the comparison theorem that {\it balls have the worst best Sobolev inequalities}
\begin{equation}
  \label{mv comparison}
  \Phi_\Om(T)\ge\Phi_B(T)\,,\qquad\forall T\in[0,\ISO(B)^{1/\psharp}]
\end{equation}
is proved. This sharp lower bound, combined with \eqref{PhiB characterization}, allows one to infer some sharp and more traditional-looking Sobolev-type inequalities, like the following sharp interpolation between \eqref{sobolev inequality on Rn} and \eqref{euclidean isoperimetry}
\begin{equation}
\label{sharp interpolation between sobolev and isop}
\frac{\|\nabla u\|_{L^p(\Om)}}{S(n,p)}+\frac{\|u\|_{L^\psharp(\pa\Om)}}{\ISO(B)^{1/\psharp}}\ge\|u\|_{L^\pstar(\Om)}\,,
\end{equation}
and the following sharp Sobolev inequality, additive in the domain of the $p$-Dirichlet energy,
\begin{equation}
  \label{domain additive sharp sobolev}
  \frac{\|\nabla u\|_{L^p(\Om)}^p}{S(n,p)^p}+\frac{\|u\|_{L^\psharp(\pa\Om)}^p}{C(n,p)}\ge\|u\|_{L^\pstar(\Om)}^p\,,
\end{equation}
which was first conjectured by Brezis and Lieb in \cite{brezisliebREMAINDER}.

\medskip

\noindent {\bf (5) Full characterization of $\Phi_H$:} In \cite[for $p=2$]{carlenloss1,carlenloss2} and, again by optimal mass transport arguments, in \cite[for $p\in(1,n)$]{maggineumayer}, it is proven that if $H$ is half-space in $\R^n$, then: for each $T\in(0,T_E)$ (Sobolev range) there is a unique $t_T\in\R$ such that
\begin{equation}
  \label{PhiH sobolev range}
  \Phi_H(T)=\|\nabla(\tau_{t_T\,\nu_H} U_S)\|_{L^p(H)}\,\Big /\,\|\tau_{t_T\,\nu_H} U_S\|_{L^\pstar(H)}\,;
\end{equation}
if $T=T_E$ (Escobar point), then
\begin{equation}
  \label{PhiH escobar}
  \Phi_H(T_E)=E(n,p)=\|\nabla(\tau_{\nu_H} U_E)\|_{L^p(H)}\,\Big /\,\|\tau_{\nu_H} U_E\|_{L^\pstar(H)}\,;
\end{equation}
for each $T>T_E$ (beyond Escobar range), there is a unique $s_T>1$ s.t.
\begin{eqnarray}
  \label{PhiH beyond E range}
  &&\Phi_H(T)=\|\nabla(\tau_{s_T\,\nu_H} U_{BE})\|_{L^p(H)}\,\Big /\,\|\tau_{s_T\,\nu_H} U_{BE}\|_{L^\pstar(H)}\,,
  \\\label{U Beyond Escobar}
  &&\mbox{where}\,\, U_{BE}(x)=(|x|^{p/(p-1)}-1)^{1-(n/p)}\,,\qquad |x|>1\,.
\end{eqnarray}
Up to the natural dilation and translation invariances, these functions are the unique minimizers of $\Phi_H(T)$.  Moreover, again by \cite{maggineumayer}: {\bf (a):} $\inf_{T\ge0}\Phi_H(T)$ is achieved at $T=T_0\in(0,T_E)$, where $t_{T_0}=0$ and $\Phi_H(T_0)=S(n,p)/2^{1/n}$; {\bf (b):} by the divergence theorem, $\Phi_H(T)> T^{\psharp}/\psharp$ for every $T>0$, and this lower bound is sharp as $T\to\infty$; {\bf (c):} finally, $\Phi_H$ has the {\it best} best Sobolev inequality, i.e.
\begin{equation}
  \label{mn comparison}
  \Phi_\Om(T)\le\Phi_H(T)\,,\qquad\forall T\ge0\,.
\end{equation}

\subsection{Statements of the main results}\label{subsection statements} With this summary on the state of the art for $\Phi_\Om$ in mind, there are two fundamental open questions that form the subject of our paper:

\medskip

{\bf Question 1:} When does $\Phi_\Om(T)$ ($T>0$) admit minimizers?

\medskip

{\bf Question 2:} Does rigidity hold in the comparison theorem \eqref{mv comparison}?

\medskip

\noindent The main idea of this paper is attacking these two closely related questions by systematically exploiting the complete characterization of $\Phi_H$ obtained in \cite{maggineumayer}.

\medskip

Concerning Question 1, a classical concentration-compactness argument characterizes the limit behavior of minimizing sequences of $\Phi_\Om(T)$ as the superposition of a standard weak limit plus at most countably many concentration points, located either in the interior of $\Om$, or on its boundary. By exploiting properties of $\Phi_H$ we are able to (i): exclude all interior concentrations and all but {\it at most one} boundary concentration, thus proving existence of minimizers for a suitable ``relaxed problem'' $\Phi_\Om^*(T)$; and (ii): completely exclude concentrations, and thus establish the existence of minimizers of $\Phi_\Om(T)$, as soon as $\pa\Om$ is of class $C^2$ and $n>2p\,$. To give precise statements, it is convenient to let $\X_\Om(T)$ denote the competition class of $\Phi_\Om(T)$, and let  $\Y_\Om(T)$ denote the set of triples $(u,\vv,\tt)$ with either $u\in \X_\Om(T)$ and $\vv=\tt=0$, or $u\in W^{1,p}(\Omega)$, $\vv\in(0,1]$, $\tt\in(0,T]$, and
\begin{equation}
    \label{u aa tt constraint}
\vv^\pstar+\int_\Om u^\pstar=1\,,\qquad \tt^\psharp+\int_{\pa\Om}u^\psharp=T^\psharp\,.
\end{equation}
The relaxed problem associated to $\Phi_\Om(T)$ is then given by
\begin{equation}
    \label{phi star Omega T}
    \Phi_\Om^*(T)=
    \inf_{\Y_\Om(T)}\,\E\,,\,\,\,\mbox{where}\,\,\,\E(u,\vv,\tt)^p=\int_\Om|\nabla u|^p+\vv^p\,\Phi_H\Big(\frac{\tt}{\vv}\Big)^p\,,
\end{equation}
with the convention that $\vv^p\Phi_H(\tt/\vv)^p=0$ if $(\vv,\tt)=(0,0)$.

\begin{theorem}[Existence of minimizers of $\Phi_\Om$]
  \label{theorem main existence}
  If $n\ge 2$, $p\in(1,n)$, and $\Om$ is a bounded open set with $C^1$-boundary in $\R^n$, then:

  \medskip

  \noindent {\bf (i):} for every $T>0$, there is a minimizer $(u,\vv,\tt)$ of $\Phi_\Om^*(T)$, and
  \begin{equation}
  \label{phiomega equal phiomegastar}
  \Phi_\Om(T)=\Phi_\Om^*(T)\,;
  \end{equation}
  moreover, if $\int_\Om u^\pstar>0$, then $u/\|u\|_{L^\pstar(\Om)}$ is a minimizer of $\Phi_\Om(\|u\|_{L^\psharp(\pa\Om)}/\|u\|_{L^\pstar(\Om)})$;

  \medskip

  \noindent {\bf (ii):} if $\Om$ has boundary of class $C^2$, $n>2p$, $T>0$, and $(u,\vv,\tt)$ is a minimizer of $\Phi_\Om^*(T)$, then $\vv=\tt=0$, and thus $u$ is a minimizer of $\Phi_\Om(T)$.
\end{theorem}

\begin{remark}
  {\rm Minimizers of $\Phi_\Om(T)$ solve the Euler--Lagrange equation \eqref{PDE of generic minimizer}. For the Euler--Lagrange equation satisfied by minimizers $(u,\vv,\tt)$ of the relaxed problem $\Phi_\Om^*(T)$, see Theorem \ref{theorem main existence part i implies} below.}
\end{remark}

Question 2 is motivated by the various rigidity statements associated to comparison theorems  in Riemannian geometry (see, e.g. \cite{cheegerebinBOOK}). In that setting, a certain model space provides a universal bound on a certain global geometric quantity (comparison theorem), which is then shown to be saturated by the model space alone (rigidity statement). With this analogy in mind, we can reformulate more precisely Question 2 as follows:

\medskip

\noindent {\bf Question 2, weak form:} Does $\Phi_\Om=\Phi_B$ on $(0,\ISO(B)^{1/\psharp})$ imply that $\Om$ is a ball?

\medskip

\noindent {\bf Question 2, strong form:} Does $\Phi_\Om(T)=\Phi_B(T)$ at just one value of $T\in(0,\ISO(B)^{1/\psharp})$ imply that $\Om$ is a ball?
\medskip

\noindent Concerning the {\it weak form} of Question 2, through a careful use of the properties of $\Phi_H$ we answer affirmatively whenever $\Om$ is bounded and connected. These conditions are optimal, as shown by unbounded or disconnected non-rigidity examples presented in \cite{maggivillaniJGA}. In fact, the argument we propose gives rigidity under the mere assumption that $\Phi_\Om=\Phi_B$ holds on an open neighborhood of $T=0$. Concerning the {\it strong form} of Question 2, which was originally formulated in \cite[Section 1.9]{maggivillaniJGA}, we can answer in the affirmative as a direct by-product of our existence result for minimizers of $\Phi_\Om(T)$ (thus, when $\Om$ has $C^2$-boundary and $n>2p$) thanks to the following ``conditional rigidity'' statement, which is proved in \cite{maggivillaniJGA} as a direct by-product of the proof of \eqref{mv comparison}:
\begin{eqnarray}
\nonumber
&&\mbox{{\it if $\Om$ is connected (possibly unbounded)}}\,,
\\
\label{mv criterion for rigidity}
&&\mbox{{\it if $\Phi_\Om(T)=\Phi_B(T)$ for a value of $T\in(0,\ISO(B)^{1/\psharp})$}}\,,
\\\nonumber
&&\mbox{{\it and if $\Phi_\Om(T)$ is known to admit minimizers (possibly just for that $T$)}}\,,
\\\nonumber
&&\mbox{{\it then $\Om$ is a ball}}\,.
\end{eqnarray}
(We notice for future use an important consequence of \eqref{mv criterion for rigidity}, namely, we have
\begin{equation}
  \label{mn comparison for balls}
  \Phi_B(T)<\Phi_H(T)\,,\qquad\forall T\in\big(0,\ISO(B)^{1/\psharp}\big)\,;
\end{equation}
indeed, by \cite{maggineumayer}, $\Phi_H(T)$ admits minimizers for every $T>0$.) With these premises, we now state our main results concerning Question 2.

\medskip

\begin{theorem}[Rigidity of ``Balls have the worst best Sobolev inequalities'']
  \label{theorem main rigidity}
  Let $n\ge 2$, $p\in(1,n)$, $\Om$  an open, bounded, connected set with $C^1$-boundary in $\R^n$, and assume that {\bf one} of the following two conditions holds:
  \medskip

  \noindent {\bf (i):} there is $T_*>0$ such that $\Phi_\Om(T)=\Phi_B(T)$ for every $T\in(0,T_*)$; or

  \medskip

  \noindent {\bf (ii):} $n>2p$, the boundary of $\Omega$ is of class $C^2$, and there is $T\in(0,\ISO(B)^{1/\psharp})$ such that $\Phi_\Om(T)=\Phi_B(T)$.

  \medskip

  \noindent Then, $\Om$ is a ball.
\end{theorem}

\subsection{Strategy of proof}\label{subsection strategy of proof} Concentration-compactness arguments and the use of sharp Sobolev-type inequalities (like the Sobolev and Escobar inequalities \eqref{sobolev inequality on Rn} and \eqref{escobar inequality}) are the standard tools of the trade in the analysis of variational problems with critical growth. As seen, if interpreted as assertions about $\Phi_H$, \eqref{sobolev inequality on Rn} and \eqref{escobar inequality} contain only very partial information (respectively, ``$\Phi_H(0)=S(n,p)$'' and ``$\Phi_H(T)\ge E(n,p)\,T$ for every $T>0$''). From this viewpoint, our arguments provide an interesting example of the potentialities of using, in the familiar context of concentration-compactness, the full characterization of $\Phi_H$ obtained in \cite{carlenloss2,maggineumayer}. We now explain how this characterization is used in this paper.

\medskip

We have already mentioned how the mere knowledge of the existence of minimizers in $\Phi_H(T)$ for every $T>0$ allows one to reduce the analysis of concentrations to the simplest possible case of a {\it single} boundary concentration (thus leading to Theorem \ref{theorem main existence}-(i)). Finer properties of $\Phi_H$ are exploited in the proof of Theorem \ref{theorem main existence}-(ii), which goes as follows. We consider the existence of a minimizer $(u,\vv,\tt)$ of $\Phi_\Om^*(T)$ with $\vv>0$, and, keeping in mind that $\Phi_\Om(T)=\Phi_\Om^*(T)$, we aim to obtain a contradiction to $\vv>0$ by constructing a competitor $v$ of $\Phi_\Om(T)$ with
\[
\int_\Om|\nabla v|^p<\int_\Om|\nabla u|^p+\vv^p\,\Phi_H\Big(\frac{\tt}\vv\Big)^p\,.
\]
We seek $v$ in the form $v=u_\e$, for the {\it Ansatz} given by
\begin{equation}
  \label{ansatz}
  u_\e(x)=(1-\vphi_\e(x))\,u(x)+\vphi_\e(x)\,\big(U^{(\e)}\circ g\big)(x)\,,\qquad x\in\Om\,.
\end{equation}
Here $x_0\in\pa\Om$ is a boundary point of $\Om$ with positive mean curvature\footnote{Our convention is that the scalar mean curvature of $\pa\Om$ is computed with respect to the outer unit normal to $\Om$, so that every bounded open set with $C^2$-boundary has at least one boundary point of positive mean curvature.}, i.e. ${\rm H}_{\pa\Om}(x_0)>0$; $\vphi_\e$ is a cut-off function between $B_{\e^\b}(x_0)$ and $B_{2\,\e^\b}(x_0)$ for $\beta=\beta(n,p)\in(0,1)$ to be suitably chosen (the condition $n>2\,p$ enters in this choice); $g$ is a boundary flattening diffeomorphism near $x_0$; and, finally, $U=U_{\tau}+b\,\e\,V_\tau$ for $\tau=\tt/\vv$, $V_\tau$ a standard perturbation of $U_\tau$, and $b$ a constant suitably chosen depending on $n$, $p$, $H_{\pa\Om}(x_0)$ and $\tau$. The energy, volume and trace expansions for $u_\e$ as $\e\to 0^+$ are computed to be
\begin{eqnarray}
\label{veps stima grad intro}
&&\int_{\Om}|\nabla u_\e|^p\le\int_\Om|\nabla u|^p+\vv^p\,\Phi_H\Big(\frac{\tt}\vv\Big)^p
\\\nonumber
&&\hspace{3cm}-\Big\{\L(U_\tau)-\frac{(n-p)}n\,\l_H(\tau)\,\M(U_\tau)\Big\}\,{\rm H}_{\pa \Om}(x_0)\,\vv^p\e+{\rm o}(\e)\,,
\\\label{veps intro correct}
&&\int_{\Om}u_\e^\pstar=1+{\rm o}(\e)\,,
\qquad \int_{\pa \Om}u_\e^\psharp=T^\psharp+{\rm o}(\e)\,,
\end{eqnarray}
where $\l_H(T)$ is the volume Lagrange multiplier of $U_T$ (see \eqref{PDE of UT} below), and where $\L$ and $\M$ are functionals defined on $U:H\to\R$ by
\begin{eqnarray}\label{def of LU}
  &&\L(U)=\int_H x_n\,|\nabla U|^p-p\,x_n\,(\pa_1 U)^2\,|\nabla U|^{p-2}\,,
  \\\label{def of Mu}
  &&\M(U)=\int_H x_n\,U^\pstar\,.
\end{eqnarray}
Modulo ${\rm o}(\e)$-perturbations of $u_\e$ aimed at correcting the volume and trace constraints to the exact values needed for inclusion in $\X_\Om(T)$, we have constructed the required competitors, and proved Theorem \ref{theorem main existence}-(ii), if we can show the existence of $c(n,p,T)>0$ such that
\begin{equation}
  \label{key inequality intro}
  \L(U_T)-\frac{(n-p)}n\,\l_H(T)\,\M(U_T)\ge c(n,p,T)\,,\qquad\forall T>0\,.
\end{equation}
Of course, the full characterization of $\Phi_H$ plays a crucial role in our proof of \eqref{key inequality intro}, see Lemma \ref{lemma key} below.

\medskip

While Theorem \ref{theorem main rigidity}-(ii) is immediate from Theorem \ref{theorem main existence}-(ii) thanks to the rigidity criterion \eqref{mv criterion for rigidity}, the proof of Theorem \ref{theorem main rigidity}-(i) requires an additional argument, which once more exploits several fine properties of $\Phi_B$ and $\Phi_H$: these include \eqref{mv criterion for rigidity}, \eqref{mn comparison for balls}, and information on the signs of the Lagrange multipliers $\l_H(T)$ and $\s_H(T)$ for minimizers $U_T$ of $\Phi_H(T)$ (see \eqref{lambdaT behavior} and \eqref{sigmaT behavior} below).

\subsection{Organization of the paper} After collecting a few preliminary results in section \ref{section prelims}, in section \ref{section boundary concentrations} we study in detail various properties of $\Phi_H$ and of its minimizers: in particular, we prove the key inequality \eqref{key inequality intro} (see Lemma \ref{lemma key}), and discuss in detail the {\it Ansatz} \eqref{ansatz} (see Lemma \ref{lemma boundary concentration}). Sections \ref{section existence generalized} and \ref{section rigidity sobolev} contain, respectively, the proofs of Theorem \ref{theorem main existence} and Theorem \ref{theorem main rigidity}. Finally, we collect some auxiliary, routine proofs in an appendix.

\medskip

\noindent {\bf Acknowledgements:} FM was supported by NSF-DMS RTG 1840314, NSF-DMS FRG 1854344, and NSF-DMS 2000034. RN was supported by NSF-DMS 2200886.

\section{Notation and preparations}\label{section prelims} Some basic notation is presented in section \ref{subsection notation}. We then discuss, in separate subsections, four useful technical lemmas:  a concentration-compactness lemma with boundary terms (Lemma \ref{lemma lions}); a second order expansion for the boundary flattening diffeomorphisms used in the {\it Ansatz} \eqref{ansatz} (Lemma \ref{lemma boundary}); some basic regularity information on minimizers of $\Phi_\Om(T)$ (Lemma~\ref{lemma what about minimizers}); and the basic technique of ``volume/trace correcting variations''  (Lemma \ref{lemma volume fix}). Some proofs are postponed to the appendix.

\subsection{Notation}\label{subsection notation} Throughout the paper we always assume that $n\ge 2$ and $p\in(1,n)$. We denote by $\L^n$ and $\H^k$ the Lebesgue measure and the $k$-dimensional Hausdorff measure of $\R^n$, although we simply set $|E|$ in place of $\L^n(E)$. We denote by $B_r(x)$ the open ball of center $x\in\R^n$ and radius $r>0$, and set $B_r=B_r(0)$, while $B$ denotes a ball of unspecified center and radius.

\medskip

Following a standard shorthand notation, by ``$f(x)=g(x)+{\rm O}_{a,b}(|x|)$ for $|x|>R$'' we mean that $|f(x)-g(x)|\le C(a,b)\,|x|$ if $|x|>R$; by ``$f(x)=g(x)+{\rm o}_{a,b}(|x|)$ as $|x|\to 0$'' we mean that $\lim_{|x|\to 0}|f(x)-g(x)|/|x|=0$ at a rate that is uniform with respect to the parameters $a$ and $b$.
\medskip

In general, we will use capital letters (e.g. $U, V, \Psi$) to denote functions defined on the half space $H$ and lowercase letters (e.g. $u,v,\vphi$) to denote functions defined on an open bounded domain $\Omega$.

\subsection{Concentration-compactness}\label{subsection cc} The following lemma is a  version of  Lions' celebrated
 concentration-compactness lemma and provides a natural starting point to study minimizing sequences of $\Phi_\Om(T)$.

\begin{lemma}[Concentration-compactness]\label{lemma lions} Let $n\ge 2$, $p\in(1,n)$,  and let $\Om\subset\R^n$ be open and bounded with $C^1$-boundary. If $\{u_j\}_j$ is a sequence in $L^1_{\rm loc}(\Om)$, $\{\nabla u_j\}_j$ is bounded in $L^p(\Om;\R^n)$ and $u_j\weak u$ as distributions in $\Om$, then the Radon measures on $\R^n$ defined by
\begin{eqnarray}\label{def of muj nuj tauj}
  \mu_j=|\nabla u_j|^p\,\L^n\llcorner\Om\,,\qquad \nu_j=|u_j|^\pstar\,\L^n\llcorner\Om\,,\qquad \tau_j=|u_j|^{\psharp}\,\H^{n-1}\llcorner\pa\Om\,,
\end{eqnarray}
have subsequential weak-star limits $\mu$, $\nu$ and $\tau$ which satisfy
\begin{eqnarray}
\label{nu}
\nu\!&=&\!\vert u\vert ^{\pstar}\mathcal{L}^n\llcorner\Om+\sum_{i\in I}\vv_i^\pstar\,\delta_{x_i}\,,
\\
\label{tau}
\tau\!&=&\!\vert u\vert ^{\psharp}\mathcal{H}^{n-1}\llcorner\partial\Om+\sum_{i\in I}\tt_i^\psharp\,\delta_{x_i}\,,
\\
\label{mu}
\mu\!&\geq&\!\vert\nabla u\vert ^p\mathcal{L}^n\llcorner\Om+\sum_{i\in I}\gg_i^p\,\delta_{x_i},
\end{eqnarray}
where $\{x_i\}_{i\in I}\subset\ov\Om$ is at most countable set, $\vv_i>0$ and $\tt_i\ge0$ for every $i\in I$,  $\tt_i>0$ only if $x_i\in\pa\Om$, and
\begin{equation}
  \label{mu lower bound}
  \gg_i\geq\vv_i\,\Phi_H\Big(\frac{\tt_i}{\vv_i}\Big)\,,\qquad\forall i\in I\,.
\end{equation}
In particular, $\gg_i\ge S\,\vv_i$ whenever $x_i\in\Om$.
\end{lemma}

\begin{proof}
  See appendix \ref{appendix lions}.
\end{proof}

\subsection{Near-boundary coordinates}\label{subsection coordinates}
In this section, we introduce two types of coordinates for a neighborhood of a boundary point of a domain $\Omega$: one that requires minimal regularity of the boundary of $\Omega$ and will suffice in the proofs of Theorem~\ref{theorem main existence}{\bf (i)} and Theorem~\ref{theorem main rigidity}{\bf (i)}, and a second that requires $C^{2}$ regularity of the boundary of $\Omega$ and will be used in the proof of  Theorem~\ref{theorem main existence}{\bf (ii)} and Theorem~\ref{theorem main rigidity}{\bf (ii)}.

\medskip

Given an open set $\Om$ with $C^1$-boundary, we denote by $\nu_\Om$ its outer unit normal and by $T_x(\pa\Om)$ the tangent space to $x\in\pa\Om$. When $\Om$ has $C^2$-boundary, we denote by ${\rm A}_{\Om}$ and ${\rm H}_\Om$ the second fundamental form and the scalar mean curvature of $\pa\Om$ defined by $\nu_\Om$.
To define coordinates near boundary points of $\Om$, for  $x\in\R^n$ we set $\pp(x)=x-x_n\,e_n$, $\DD_r=\{x:x_n=0\,,|\pp x|<r\}$, and $\CC_r=\big\{x:|x_n|<r\,,|\pp(x)|<r\big\}$. In particular, if $\Om$ is an open set with $C^1$-boundary such that
\begin{equation}
  \label{Omega C1 hp 1}
  0\in\pa\Om\,,\qquad T_0\,(\pa\Om)=\{x_n=0\}\,,\qquad\nu_{\Om}(0)=-e_n\,,
\end{equation}
then we can find $r_0>0$ and $\ell:\DD_{r_0}\to(-r_0,r_0)$ such that $\ell(0)=0$, $\nabla\ell(0)=0$, and
\begin{eqnarray*}
\Om\cap\CC_{r_0}\!\!\!&=&\!\!\!\big\{x+t\,e_n:x\in\DD_{r_0}\,,r_0>t>\ell(x)\big\}\,,
\\
(\pa\Om)\cap\CC_{r_0}\!\!\!&=&\!\!\!\big\{x+\ell(x)\,e_n:x\in\DD_{r_0}\big\}\,.
\end{eqnarray*}
We then define the maps $F:\DD_{r_0}\to\pa\Om$, $f:\CC_{r_0}\to\R^n$ and $\hat{f}:\CC_{r_0}\to\R^n$ by setting
\begin{eqnarray}\label{def of F}
F(x)\!\!\!&=&\!\!\!x+\ell(x)\,e_n\,,\qquad \qquad  \qquad \quad \, x\in\DD_{r_0}\,,
\\
\label{def of f hat}
\hat{f}(x)\!\!\!&=&\!\!\!F(\pp x)+x_ne_n \,,\qquad \qquad \qquad \!  x\in\CC_{r_0}\,.
\\
\label{def of f}
f(x)\!\!\!&=&\!\!\!F(\pp x)-x_n\,\nu_\Om(F(\pp x))\,,\qquad x\in\CC_{r_0}\,.
\end{eqnarray}
In this way, for every $y\in(\pa\Om)\cap\CC_{r_0}$, if we set $y=F(\pp x)$, then
\[
\nu_\Om(y)=\frac{\nabla\ell(x)-e_n}{\sqrt{1+|\nabla\ell(x)|^2}}\,,\qquad
{\rm H}_\Om(y)=\Div\Big(\frac{\nabla\ell}{\sqrt{1+|\nabla\ell|^2}}\Big)(x)\,.
\]

Notice that the map $f$ need not be of class $C^1$ if the boundary of $\Omega$ is only of class $C^1$, while the map $\hat{f}$ will be as regular as the boundary of $\Omega$. The following lemma summarizes basic properties about the map $\hat{f}$.

\begin{lemma}[Near-boundary coordinates, one]
  \label{lemma boundary 0}
   If $H=\{x_n>0\}$, $\Om$ is an open set with $C^1$-boundary and \eqref{Omega C1 hp 1} holds,
   then there exist $r_0$ and $C_0$ positive such that the map $\hat{f}$ in \eqref{def of f hat} defines a $C^1$-diffeomorphism from $\CC_{r_0}$ to its image,
   taking $\DD_{r_0}$ into $\pa\Om$ and with
  \begin{eqnarray}
    \label{f hat inclusions}
    \hat{f}(\CC_{r/C_0}\cap H)\subset \Om\cap B_r\subset \hat{f}(\CC_{C_0\,r}\cap H)\qquad\forall r<r_0/C_0\,.
  \end{eqnarray}
  Moreover, letting $\hat{g}=\hat{f}^{-1}$ denote the inverse of $\hat{f}$, we have
  \begin{eqnarray}\label{f C1 estimates}
    &&\nabla \hat{f}=\Id_{\R^n}+{\rm o}(1)\,,\qquad (\nabla \hat{g})^*\circ \hat{f}=\Id_{\R^n}+{\rm o}(1)\,,
    \\
    \label{f C1 estimate for Jf}
    &&J\hat{f}=1+{\rm o}(1)\,,\qquad \qquad  1\le J^{\pa H}\hat{f}\le 1+{\rm o}(1)\,,\quad\quad \mbox{for $x\in\CC_{r_0}$.}
  \end{eqnarray}
  The orders in \eqref{f C1 estimates} and \eqref{f C1 estimate for Jf} depend on $\pa\Om$ and on $0\in\pa\Om$.
\end{lemma}

\begin{proof}
	See appendix \ref{appendix boundary}.
\end{proof}

The map $f$ defined in \eqref{def of f} has the advantage that, when the boundary of $\Omega$ is at least of class $C^2$,  curvature quantities appear in expansions of the metric coefficients and the volume form in these coordinates. These properties are the content of the following lemma.

\begin{lemma}[Near-boundary coordinates, two]
  \label{lemma boundary}
   If $H=\{x_n>0\}$, $\Om$ is an open set with $C^{2}$-boundary and \eqref{Omega C1 hp 1} holds,
   then there exist $r_0$ and $C_0$ positive such that the map $f$ in \eqref{def of f} defines a $C^1$-diffeomorphism from $\CC_{r_0}$ to its image,
   taking $\DD_{r_0}$ into $\pa\Om$ and with
  \begin{eqnarray}
    \label{f inclusions}
    f(\CC_{r/C_0}\cap H)\subset \Om\cap B_r\subset f(\CC_{C_0\,r}\cap H)\qquad\forall r<r_0/C_0\,.
  \end{eqnarray}
  Moreover, for $x\in\CC_{r_0}$ and $x \in \DD_{r_0}$ respectively, we have
  \begin{eqnarray}
  \label{f C2 estimate for Jacobian}
  &&\!\!\!\!\!\!\!\!\!\!\!\!\!\!\!\!\!\!\!\!\!\!Jf(x)=1-x_n\,{\rm H}_\Om(0)+{\rm O}(|x|^2)\,,
   \qquad 1 \le J^{\pa H} f (x) \le 1+{\rm O}(|x|^2)\,,
  \end{eqnarray}
  and if $\{e_i\}_{i=1}^{n-1}$ is an orthonormal basis of $\R^{n-1}\subset \R^n$ of eigenvectors of $\nabla^2\ell(0)$ and $\{\k_i\}_{i=1}^{n-1}$ denote the corresponding eigenvalues (so that, by \eqref{Omega C1 hp 1}, they are the principal curvatures of $\pa\Om$ with respect to $\nu_\Om$ computed at $0\in\pa\Om$, and in particular ${\rm H}_\Om(0)=\sum_{i=1}^{n-1}\k_i$), then, letting $g=f^{-1}$ denote the inverse of $f$, we have
  \begin{eqnarray}\label{f C2 estimate for nabla g}
  &&\!\!\!\!\!\!\!\!\!\!\!\!\!\!\!\!\!\!\!\!\!\!\!(\nabla g)^*\circ f=\Id_{\R^n}\!+\big(\nabla\ell\otimes e_n\!-\!e_n\otimes\nabla\ell)+x_n\sum_{i=1}^{n-1}\k_ie_i\otimes e_i+\!{\rm O}(|x|^2)\,.
  \end{eqnarray}
 The orders in \eqref{f C2 estimate for Jacobian} and  \eqref{f C2 estimate for nabla g} depend on $\pa\Om$ and on $0\in\pa\Om$.
\end{lemma}

\begin{proof}
  See appendix \ref{appendix boundary}.
\end{proof}

\begin{remark}\label{remark fx0}
  {\rm Given $x_0\in\pa\Om$, we denote by $\pi_{x_0}$ the rigid motion of $\R^n$ that maps $x_0$ to $0$ such that \eqref{Omega C1 hp 1} holds with $\pi_{x_0}(\Om)$ in place of $\Om$. Then we set, for $\hat{f}$ and  $f$ defined as in \eqref{def of f hat} and \eqref{def of f} respectively but with $\pi_{x_0}(\Om)$ in place of $\Om$,
  \[
  \hat{f}_{x_0}=\pi_{x_0}^{-1}\circ \hat{f}\,, \qquad   f_{x_0}=\pi_{x_0}^{-1}\circ f\,.
  \]
  Clearly
  these maps are diffeomorphisms on $\CC_{r_0}$, mapping $H\cap\CC_{r_0}$ into a neighborhood of $x_0$ in $\Om$ and $\DD_{r_0}=(\pa H)\cap\CC_{r_0}$ into a neighborhood of $x_0$ in $\pa\Om$, and satisfies proper reformulations of the estimates in Lemmas \ref{lemma boundary 0} and  \ref{lemma boundary}. Here $r_0$ and $C_0$ depend also on the choice of $x_0$, and can of course be assumed uniform across $x_0\in\pa\Om$ if $\pa\Om$ is bounded.}
\end{remark}

\medskip

\subsection{Properties of minimizers}\label{subsection minimizers of PhiOmega} The following lemma gathers some fundamental properties of minimizers of $\Phi_\Om$ that will be needed in the sequel.

\begin{lemma}\label{lemma what about minimizers}
  If $n\ge 2$, $p\in(1,n)$, $T>0$, $\Om$ is a bounded open set with $C^2$-boundary, and $u$ is a minimizer of $\Phi_\Om(T)$, then $u$ is bounded and Lipschitz continuous in $\Omega$, there are $\l$ and $\s$ such that the Euler--Lagrange equation \eqref{PDE of generic minimizer} holds in the weak sense, and the balance condition
  \begin{equation}
    \label{balance condition}
    \l\,\int_\Om u^{\pstar-1}+\s\,\int_{\pa\Om} u^{\psharp-1}=0\,,
  \end{equation}
  holds.
\end{lemma}

\begin{proof} By a standard argument, based on similar considerations to the one presented in Lemma~\ref{lemma volume fix} below, one sees that a minimizer $u$ of $\Phi_{\Omega}(T)$ is a $W^{1,p}(\Om)$-distributional solution of the Euler-Lagrange equation \eqref{PDE of generic minimizer} for some $\l,\s\in\R$. As soon as $\Omega$ is bounded and has Lipschitz boundary, one can exploit \eqref{PDE of generic minimizer} in conjunction with a Moser iteration argument to prove that $u \in L^\infty(\Om)$ (see, e.g. \cite[Theorem 3.1]{MarinoWinkertMoser}; their result applies to \eqref{PDE of generic minimizer} by taking, in the notation of their paper,  $\mathcal{A}(x,u,\na u ) =|\na u|^{p-2}\na u$, $\mathcal{B}(x,u,\na u ) = \lambda u^{\pstar-1}$, and $\mathcal{C}(x,u) =\sigma u^{\psharp -1}$). On further assuming that $\pa \Omega$ is of class $C^2$, then the classical result \cite[Theorem 1.7]{Lieberman1992} can be applied to deduce that $u \in C^{1,\beta}(\overline{\Omega})$ for a suitable $\beta = \beta(n,p)\in(0,1)$ (for more details, see \cite[Theorem 3.9]{MarinoWinkertMoser}). In particular, $u$ is bounded and Lipschitz continuous on $\Om$, as claimed.
\end{proof}

\subsection{Volume/trace correcting variations}\label{subsection volume fix} At various stages in our arguments we will need to slightly modify certain competitors so to restore the volume and trace constraints defining $\X_\Om(T)$. The following lemma describes the basic mechanism used to this end.

\begin{lemma}[Volume/trace correcting variations]
  \label{lemma volume fix}
  If $n\ge 2$, $p\in(1,n)$, $M>0$, $\Om$ is an open set with $C^1$-boundary, $v\in L^1_{\rm loc}(\Om)$ with $\nabla v\in L^p(\Om;\R^n)$, and if $x_0\in\R^n$ and $r>0$ are such that
  \begin{equation}
    \label{are positive}
  \mbox{$\int_{\Om\setminus B_r(x_0)}v^\pstar$ and $\int_{(\pa\Om)\setminus B_{r}(x_0)} v^\psharp$ are positive and finite}\,,
  \end{equation}
  then there exist positive constants $\eta$ and $C$, and functions $\vphi\in C^\infty_c(\R^n\setminus\ov{B}_r(x_0))$ and $\psi\in C^\infty_c(\Om\setminus\ov{B}_r(x_0))$, all depending on $n$, $p$, $v$ and $M$ only, and with the following property.

  \medskip

  If $\{v_\e\}_{\e<\e_0}\subset L^1_{{\rm loc}}(\Om)$ is such that, for every $\e<\e_0$,
  \begin{equation}
    \label{hps on veps}
      \mbox{$v_\e=v$ on $\Om\setminus B_r(x_0)$}\,,\qquad\int_\Om|\nabla v_\e|^p\le M\,,
  \end{equation}
  then for every $(a,b)$ with $|a|,|b|<\eta/C$, we can find $(s,t)$ with $|s|,|t|<\eta$ such that
  \[
  w_\e=v_\e+s\,\vphi+t\,\psi
  \]
  satisfies
  \begin{equation}
    \label{vt fix trace and volume fixed}
    \int_{\pa\Om} |w_\e|^\psharp =a+\int_{\pa\Om}|v_\e|^\psharp\,,\qquad  \int_\Om |w_\e|^\pstar=b+\int_\Om|v_\e|^\pstar\,,
  \end{equation}
  \begin{equation}
    \label{vt fix cost for the gradient}
    \Big|\int_\Om|\nabla w_\e|^p-\int_\Om|\nabla v_\e |^p\Big|\le C\,\big(|a|+|b|\big)\,.
  \end{equation}
\end{lemma}

\begin{proof}
  By \eqref{are positive} there are  $\xi\in C^\infty_c(\R^n\setminus\ov{B}_r(x_0))$ and $\psi\in C^\infty_c(\Om\setminus\ov{B}_r(x_0))$ such that
  \begin{equation}
    \label{vt fix choice 1}
      \int_{\pa\Om}v^{\psharp-1}\,\xi =1\,,\qquad\int_{\Om}v^{\pstar-1}\,\psi=1\,.
  \end{equation}
  Setting $\vphi=\xi-(\int_\Om v^{\pstar-1}\xi)\,\psi$, we have $\vphi\in C^\infty_c(\R^n\setminus\ov{B}_r(x_0))$ with
  \begin{equation}
        \label{vt fix choice 2}
          \int_{\Om}v^{\pstar-1}\,\vphi =0\,,\qquad  \int_{\pa\Om}v^{\psharp-1}\,\vphi =1\,.
  \end{equation}
  We now define $h_\e:\R^2\to\R^2$ by
  \[
  h_\e(s,t)\!=\!\Big(\int_{\pa\Om} |v_\e+s\,\vphi+t\,\psi|^\psharp -\!\int_{\pa\Om}|v_\e|^\psharp  ,\ \int_\Om |v_\e+s\,\vphi+t\,\psi|^\pstar -\!\int_\Om|v_\e|^\pstar  \Big)\,.
  \]
  By \eqref{hps on veps} we have $h_\e\in C^{1,\a}(\R^2;\R^2)$ for some $\a=\a(n,p)\in(0,1)$, with
  \[
  \sup_{\e<\e_0}\|h_\e\|_{C^{1,\a}(\R^2;\R^2)}<\infty\,;
  \]
  moreover, $h_\e(0,0)=0$ and, by \eqref{vt fix choice 1} and \eqref{vt fix choice 2},
  \[
  \nabla h_\e(0,0)=\left(\begin{array}{c c c}
    \psharp\,\int_{\pa\Om}|v_\e|^{\psharp-1}\,\vphi \, & &\psharp\,\int_{\pa\Om}|v_\e|^{\psharp-1}\,\psi \,
    \\
    &
    \\
    \pstar\,\int_{\Om}|v_\e|^{\pstar-1}\,\vphi &    & \pstar\,\int_{\Om}|v_\e|^{\pstar-1}\,\psi\,
      \end{array}\right)=\Big(\begin{array}{c c}
    \psharp & 0
    \\
    0 & \pstar
  \end{array}\Big)\,.
  \]
  We can thus apply the inverse function theorem {\it uniformly} in $\e$, to find positive constants $\eta$ and $C_1$  depending on $n$, $p$, and $v$ so that each $h_\e$ is invertible on $E=\{(s,t):|s|,|t|<\eta\}$, with $\{(a,b):|a|,|b|<\eta/C_1\}\subset h_\e(E)$, and $\nabla h_\e^{-1}(a,b)\ge\Id_{2\times 2}/C_1$ (in the sense of positive definite matrices) for every $(a,b)\in h_\e(E)$. In particular, if we let $(s,t)=h_\e^{-1}(a,b)$ for a pair $(a,b)$ with $|a|,|b|<\eta/C_1$, then the function $w_\e=v_\e+s\,\vphi+t\,\psi$ satisfies \eqref{vt fix trace and volume fixed} and $|(a,b)|=|h_\e^{-1}(s,t)|\ge|(s,t)|/C_1$. Moreover, by the elementary inequality
  \[
  \big||X+Y|^p-|X|^p\big|\le p\,\max\big\{|X|,|Y|\big\}^{p-1}\,|Y|\,\qquad\forall X,Y\in\R^n\,,
  \]
  we see that, setting $\g=\max\big\{|\nabla v_\e|,|\nabla\vphi|,|\nabla\psi|\big\}^p$,
  \begin{equation}\label{eqn: const fix energy term}
  	\begin{split}
  	 \Big|\int_\Om|\nabla w_\e|^p&-\int_\Om|\nabla v_\e|^p\Big|\le p\,\int_\Om \g^{(p-1)/p}\,\big(|s|\,|\nabla\vphi|+|t|\,|\nabla\psi|\big)
  \\
  &\le C\,\Big(\int_\Om \g \, \Big)^{(p-1)/p}\,\Big(\int_\Om \big(|s|^p\,|\nabla\vphi|^p+|t|^p\,|\nabla\psi|^p\big) \Big)^{1/p}
  \\
  &\le C\,|(s,t)|\le C_2\,|(a,b)|\,,
  	\end{split}
  \end{equation}
  for a constant $C_2$ depending on $n$, $p$, $v$, and $M$. Letting $C = \max \{ C_1, C_2\}$ concludes the proof of the lemma.
\end{proof}

\section{Boundary concentrations}\label{section boundary concentrations}

\subsection{Properties of $\Phi_H$-minimizers}\label{subsection PhiH} We recall some facts proved in \cite{maggineumayer} about $\Phi_H$ and its minimizers. Recall that we denote by $T_0$ the minimum point of $\Phi_H$, so that
\begin{equation}
\label{properties of T0}
T_0\in(0,T_E)\,,\qquad \Phi_H(T_0)=2^{-1/n}\,S(n,p)\,,
\end{equation}
where $T_E$ is the ``Escobar trace'' defined in \eqref{def of TE}. If we set $H=\{x_n>0\}$, the minimizers of $U_T$ of $\Phi_H(T)$ for $T>0$ are characterized (modulo the obvious scaling and translation invariance of $\Phi_H$) as
\begin{eqnarray}
  \label{UT}
  U_T(x)=c_T\!\!\left\{
\begin{split}
  &\tau_{t_T\,e_n}U_S(x)=(1+|x-t_T\,e_n|^{p'})^{1-(n/p)}\,,\qquad\!\!\!\!\hspace{0.3cm}T\in(0,T_E)\,,
  \\
  &\tau_{e_n}U_E(x)=|x+e_n|^{(p-n)/(p-1)}\,,\qquad\!\!\!\!\hspace{1.7cm} T=T_E\,,
  \\
  &\tau_{s_T\,e_n}U_{BE}(x)=(|x-s_T\,e_n|^{p'}-1)^{1-(n/p)}\,,\!\!\!\!\qquad T>T_E\,,
\end{split}
\right .
\end{eqnarray}
where the constants $c_T$, $t_T$ and $s_T$ are chosen in such a way that
\begin{equation}
  \label{UT volume trace and energy}
  \int_H\, U_T^{p^\star}=1\,,\qquad  \int_{\pa H}\, U_T^{p^\sharp}=T^{p^\sharp}\,,\qquad\int_H|\nabla U_T|^p=\Phi_H(T)^p\,.
\end{equation}
It is convenient to keep in mind that the various formulas for $U_T$ listed in \eqref{UT} all share the same decay behavior at infinity, that is (see \eqref{UT decay pointwise} below), we have
\begin{equation}\label{eqn: decay first part}
U_T(x) \sim |x|^{(p-n)/(p-1)}, \qquad |\na U_T| \sim |x|^{(1-n)/(p-1)} \qquad \text{ as }  |x|\to \infty	.
\end{equation}
(where the rate depends on the specific value of $T$ under consideration). The constants $t_T$ and $s_T$ have the following properties: $T\in(0,T_E)\mapsto t_T$ is continuous and strictly decreasing, with $t_T>0$ if and only if $T\in (0,T_0)$, and
\begin{equation}
  \label{properties of tT}
  \lim_{T\to 0^+}t_T=+\infty\,,\qquad \lim_{t\to (T_E)^-}t_T=-\infty\,,\qquad t_{T_0}=0\,,
\end{equation}
while $T\in(T_E,\infty)\mapsto s_T$ is continuous, negative, strictly increasing, with
\begin{equation}
  \label{properties of sT}
\lim_{T\to (T_E)^+}s_T=-\infty\,,\qquad\lim_{T\to+\infty}s_T=-1\,.
\end{equation}
Denoting by $\Delta_pv=\Div(|\nabla v|^{p-2}\nabla v)$ the $p\,$-Laplace operator, we have
\begin{equation}
  \label{PDE of UT}
  \left\{
  \begin{split}
    &-\Delta_p U_T=\l_H(T)\,U_T^{p^\star-1}\,\quad\hspace{1cm}\mbox{on $H$}\,,
    \\
    &|\nabla U_T|^{p-2}\,\frac{\pa U_T}{\pa\nu_H}=\s_H(T)\,U_T^{p^\sharp-1}\,\quad\mbox{on $\pa H$}\,,
  \end{split}\right .\qquad\forall T>0\,,
\end{equation}
where $\l_H,\s_H:(0,\infty)\to\R$ are continuous and satisfy the relations
\begin{equation}\label{lambdaT sigmaT basic identities}
  \Phi_H(T)^p=\l_H(T)+\s_H(T)\,T^\psharp\,,\qquad  \s_H(T)=\frac{\Phi_H(T)^{p-1}\,\Phi_H'(T)}{T^{\psharp-1}}\,,
\end{equation}
(see \cite[Lemma 3.3]{maggineumayer}\footnote{Notice that in \cite[(3.16)]{maggineumayer} it is incorrectly stated that $\s_H(T)=\Phi_H(T)^{p-1}\,\Phi_H'(T)/(\psharp\,T^{\psharp-1})$, where the extra $1/\psharp$-factor is wrongly introduced in the penultimate displayed equation in the proof of Lemma 3.3, where $\tau'(0)=T^{1-\psharp}/\psharp$ should be replaced by $\tau'(0)=T^{1-\psharp}$. This error is inconsequential for the arguments in \cite{maggineumayer}, since this expression for $\s_H(T)$ is only used in equation (3.22) and subsequent displayed equations, and since, in all these subsequent identities, a generic multiplicative factor $c(n,p)$ is used (in particular, two functions of $(n,p)$ differing by a $1/\psharp$-factor are both $c(n,p)$). It will instead be important in the proof of \eqref{multipliers gen min equation} to work with correct expression for $\s_H(T)$.}) as well as
\begin{eqnarray}
  \label{sigma limits}
  \lim_{T\to 0^+}\s_H(T)=-\infty\,,\qquad \lim_{T\to +\infty}\s_H(T)=+\infty
\\
\label{lambda limits}
\lim_{T\to 0^+}\l_H(T)>0\,,\qquad \lim_{T\to+\infty}\l_H(T)=-\infty\,.
\end{eqnarray}
The signs of $\s_H$ and $\l_H$ can be easily deduced from \eqref{UT}, and satisfy
\begin{eqnarray}
\label{sigmaT behavior}
&&(0,T_0)=\{\s_H<0\}\,,\qquad  (T_0,\infty)=\{\s_H>0\}\,,\qquad\s_H(T_0)=0\,,
\\
\label{lambdaT behavior}
&&(0,T_E)=\{\l_H>0\}\,,\qquad (T_E,\infty)=\{\l_H<0\}\,,\qquad\l_H(T_E)=0\,.
\end{eqnarray}
\medskip

\subsection{A key inequality and further properties of $U_T$}  In this section, we prove the key inequality \eqref{key inequality} for the functions $\L$ and $\M$ introduced in \eqref{def of LU} and \eqref{def of Mu}, namely
\begin{eqnarray*}
  &&\L(U)=\int_H x_n\,|\nabla U|^p-p\,x_n\,(\pa_1 U)^2\,|\nabla U|^{p-2}\,,
  \\
  &&\M(U)=\int_H x_n\,U^\pstar\,.
\end{eqnarray*}
Whenever $U$ satisfies the decay properties \eqref{eqn: decay first part} (e.g., when $U$ is a compactly supported perturbation of some $U_T$), we have that $\M(U)<\infty$; however, $\L(U)<\infty$ under \eqref{eqn: decay first part} if and only if $n>2\,p-1$; see \eqref{UT decay pstar first moment}  and \eqref{UT decay grad p first moment} below.

\begin{lemma}[Key inequality]
  \label{lemma key} If $n\ge 2$, $p\in(1,n)$, $n>2p-1$, and $T>0$, then there is a positive constant $c(n,p,T)$ such that
  \begin{eqnarray}
  \label{key inequality}
   \L(U_T) -\frac{n-p}n\,\l_H(T)\,\M(U_T) \geq c(n,p,T)\,.
  \end{eqnarray}
\end{lemma}

The following lemma will be useful in proving Lemma \ref{lemma key}.

\begin{lemma}\label{lemma 3.17}
	If $H=\{x_n>0\}$, $U:H \to \R$ is radially symmetric with respect to $t \,e_n$ for some $t \in \R$, and $\int_H x_n |\na U|^p$ is finite, then
	\[
	\int_{H}x_n\,|\nabla U|^{p-2}\big\{(\pa_n U)^2-(\pa_1 U)^2\big\}>0\,.
	\]
\end{lemma}

\begin{proof}[Proof of Lemma \ref{lemma 3.17}] We have  $U(x)=\eta(|x-t\,e_n|)$, $y=x-t\,e_n$, $r=|y|$, and $\hat{y}=y/|y|$, so that
\[
x_n\,|\nabla U|^{p-2}\big\{(\pa_n U)^2-(\pa_1 U)^2\big\}
=(y_n+t)\,|\eta'(r)|^p\big\{(\hat{y}_n)^2-(\hat{y}_1)^2\big\}\,,
\]
and, setting $y=r\,z$,
\begin{eqnarray*}
&&\int_{H}x_n\,|\nabla U|^{p-2}\big\{(\pa_n U)^2-(\pa_1 U)^2\big\}
=\int_{\{y_n>-t\}}(y_n+t)\,|\eta'(r)|^p\big\{(\hat{y}_n)^2-(\hat{y}_1)^2\big\}
\\
&&=\int_0^\infty\,|\eta'(r)|^p\,dr  \int_{\{y_n>-t\}\cap \pa B_r}(y_n+t)\,\big\{(\hat{y}_n)^2-(\hat{y}_1)^2\big\}\,d\H^{n-1}_y
\\
&&=\int_0^\infty\,r^n\,|\eta'(r)|^p\,dr  \int_{\{z_n>-t/r\}\cap \pa B_1}\Big(z_n+\frac{t}r\Big)\,\big\{z_n^2-z_1^2\big\}\,d\H^{n-1}_z\,.
\end{eqnarray*}
We conclude the proof by showing that
\begin{equation}
  \label{fine 3}
  \int_{\{z_n>-s\}\cap \pa B_1}\big(z_n+s\big)\,\big\{z_n^2-z_1^2\big\}\,d\H^{n-1}_z>0 \quad \text{ for all } s\in(-1,1),
\end{equation}
noting that for each $s \in [1,\infty)$, this integral vanishes by symmetry while for $s \in (-\infty, -1]$ the domain of integration is empty.  To see \eqref{fine 3},  let $\pp:\R^n\to\R^{n-2}$ denote the projection map $\pp(x)=(x_2,...,x_{n-1})$ (if $n=2$ there is no need to introduce $\pp$). The tangential coarea factor of $\pp$ along $\pa B_1$ defines a positive function $K:\pa B_1\to(0,\infty]$ which is $\H^{n-1}$-a.e. finite on $\pa B_1$, and which is {\it independent} of the variables $(x_1,x_n)$, i.e. $K(x_1,w,x_n)=K(w)$ for every $(x_1,w,x_n)\in\pa B_1$. Therefore, setting for brevity $M_s=\{z_n>-s\}\cap \pa B_1$,
\[
\int_{M_s}\big(z_n+s\big)\,\big\{z_n^2-z_1^2\big\}\,d\H^{n-1}_z
=\int_{\pp(M_s)}\,\frac{d\L^{n-2}_w}{K(w)}\int_{M_s\cap\pp^{-1}(w)}\big(z_n+s\big)\,\big\{z_n^2-z_1^2\big\}\,d\H^1_{(z_1,z_n)}\,,
\]
where
\[
M_s\cap\pp^{-1}(w)=\Big\{(z_1,z_n):z_1^2+z_n^2=1-|w|^2\,,z_n>-s\Big\}\,,\qquad s\in(-1,1)\,.
\]
As before, the inner integral above vanishes when $1-|w|^2\leq s$ by symmetry and the domain of integration is empty when $-s \geq 1-|w|^2$.
We are thus left to prove that
\[
  \int_{-\a}^{\pi+\a}\big(\sin\theta+\sin\a\big)\,\big\{\sin^2\theta-\cos^2\theta\big\}\,d\theta>0
\]
  for $\a\in(0,\pi/2)$ (corresponding to the case when $s\geq0$) and
  \[
    \int_{\a}^{\pi-\a}\big(\sin\theta-\sin\a\big)\,\big\{\sin^2\theta-\cos^2\theta\big\}\,d\theta>0
  \]
   for $\a\in(0,\pi/2)$ (corresponding to the case when $s<0$). Direct computation shows that both of these integrals are equal to the positive quantity $(2/3)\,(\cos\alpha)^3$.
\end{proof}

We now prove Lemma~\ref{lemma key}.

\begin{proof}[Proof of Lemma \ref{lemma key}] Testing \eqref{PDE of UT} with $x_n\,U_T$ we find
\begin{eqnarray*}
  \l_H(T)\,\int_H x_n\,U_T^\pstar=-\int_Hx_n\,U_T\,\Delta_p U_T
  =\int_H|\nabla U_T|^{p-2}\,\nabla U_T\cdot\nabla (x_n\,U_T)\,.
\end{eqnarray*}
Here the integration by parts is justified since $x_n=0$ on $\pa H$ and since by \eqref{eqn: decay first part},
\[
\Big|\int_{H\cap\pa B_R} x_n\,U_T\,|\nabla U_T|^{p-2}(\nu_{B_R}\cdot\nabla U_T)\Big|\le C\,\frac{R^{n-1}\,R}{R^{(n-p)/(p-1)}}\,
\Big(\frac1{R^{(n-1)/(p-1)}}\Big)^{p-1}\to 0
\]
like $R^{-(n+1)/(p-1)}$ as $R\to\infty$.
We thus find that
\begin{eqnarray*}
  &&\!\!\!\!\!\!\!\!\!\!\L(U_T)-\frac{n-p}n\,\l_H(T)\,\M(U_T)=
  \int_{H}x_n\,|\nabla U_T|^p-p\,x_n\,(\pa_1 U_T)^2\,|\nabla U_T|^{p-2}
  \\&&\hspace{4cm}-\frac{n-p}n\,\Big\{\int_H x_n\,|\nabla U_T|^p+\int_H U_T\,|\nabla U_T|^{p-2}\,\pa_n U_T\Big\}\,.
\end{eqnarray*}
Now, since $|\nabla U_T|$ is symmetric by reflection with respect to the hyperplanes $\{x_i=0\}$, $i=1,...,n-1$, we see that
\begin{eqnarray*}
  \int_H x_n\,|\nabla U_T|^p&=&\sum_{i=1}^{n-1}\int_H\,x_n\,(\pa_iU_T)^2|\nabla U_T|^{p-2}+\int_H\,x_n\,(\pa_nU_T)^2|\nabla U_T|^{p-2}
  \\
  &=&(n-1)\int_H\,x_n\,(\pa_1U_T)^2|\nabla U_T|^{p-2}+\int_H\,x_n\,(\pa_nU_T)^2|\nabla U_T|^{p-2}
\end{eqnarray*}
so that continuing from above we have
\begin{eqnarray*}
  &&\L(U_T)-\frac{n-p}n\,\l_H(T)\,\M(U_T)
  \\
  &&=\frac{p}n\,\int_{H}x_n\,|\nabla U_T|^p-p\,\int_H\,x_n\,(\pa_1 U_T)^2\,|\nabla U_T|^{p-2}
  -\frac{n-p}n\,\int_H U_T\,|\nabla U_T|^{p-2}\,\pa_n U_T
  \\
  &&
  =\frac{p}n\,\int_{H}x_n\,|\nabla U_T|^{p-2}\Big\{(\pa_n U_T)^2-(\pa_1 U_T)^2\Big\}
  +\frac{n-p}n\,\int_H U_T\,|\nabla U_T|^{p-2}\,\big(-\pa_n U_T\big)\,.
\end{eqnarray*}
In particular, the lemma is proved by showing that
\begin{eqnarray}
  \label{fine 2}
  &&\int_{H}x_n\,|\nabla U_T|^{p-2}\big\{(\pa_n U_T)^2-(\pa_1 U_T)^2\big\}> 0\,,\\
  \label{fine 1}
  &&\int_H U_T\,|\nabla U_T|^{p-2}\,\big(-\pa_n U_T\big)>0\,,
\end{eqnarray}
where the first inequality, \eqref{fine 2}, is immediate from Lemma~\ref{lemma 3.17} (recall that $n>2p-1$ and  $U_T$ is radially symmetric with respect to $t e_n$ for some $t\in\R$).

\medskip

We are thus left to prove \eqref{fine 1}. This is immediate in the case when $T\geq T_0$, because in that case, by \eqref{properties of tT} and \eqref{properties of sT}, $U_T$ has center of symmetry at $t \, e_n$ for some $t \leq 0$, and thus $\pa_n U_T<0$ on $H$. By \eqref{properties of tT}, if $T \in (0, T_0)$, then $U_T$ has center of symmetry at $t\,e_n$ for some $t>0$. Correspondingly, $U_T\,\pa_nU_T$ is odd with respect to $\{x_n=t\}$, with $U_T\,\pa_nU_T<0$ on $\{x_n>t\}$ and $U_T\,\pa_nU_T>0$ on $\{0<x_n<t\}$: in particular, if $p_t$ denotes the reflection with respect to $\{x_n=t\}$, then
\begin{eqnarray*}
&&\int_{2t>x_n>t}\,x_n\,U_T\,(-\pa_n U_T)=\int_{t>x_n>0}\,(p_t(x)\cdot e_n)\,\,[U_T\,(-\pa_n U_T)](p_t(x))\,dx
\\
&&=\int_{t>x_n>0}\,(p_t(x)\cdot e_n)\,\,[U_T\,\pa_n U_T)](x)\,dx
\ge\int_{t>x_n>0}\,x_n\,\,U_T\,\pa_n U_T\,,
\end{eqnarray*}
so that
\[
\int_H U_T\,|\nabla U_T|^{p-2}\,\big(-\pa_n U_T\big)\ge
\int_{\{x_n>2t\}} U_T\,|\nabla U_T|^{p-2}\,\big(-\pa_n U_T\big)\,,
\]
and the latter integral is positive because $\pa_n U_T<0$ on $\{x_n>t\}$.
\end{proof}

\subsection{Standard variations of $\Phi_H$-minimizers}\label{sec: standard variations}
We now introduce a ``class of standard variations'' of minimizers of $\Phi_H$. With $H=\{x_n>0\}$, we define $\zeta=\zeta(r,T):[0,\infty)\times(0,\infty)\to[0,\infty)$, so that, setting $V_T(x)=\zeta(|x-e_n|,T)$ for $x\in\R^n$, we have
\begin{equation}
\label{def of VT}
    V_T\in C^\infty_c(H;[0,\infty))\,,\qquad \int_H\,U_T^{\pstar-1}\,V_T=1\,,\qquad\int_{\pa H}\,U_T^{\psharp-1}\,V_T=0\,.
\end{equation}
Given $T>0$ we denote by
\begin{equation}
  \label{def of U T eps}
  \U_T
\end{equation}
the family of functions $U:H\to\R$ of the form
\[
U=U_T+t\,V_T\,,\qquad |t|\le 1\,.
\]
The following lemma contains some basic properties of functions in $\U_T$. We notice that
\begin{equation}\label{symmetry of U in UTeps}
  \begin{split}
    &\mbox{every $U\in\U_T$ is symmetric by reflection}
    \\
    &\mbox{with respect to the coordinates $x_1$, ..., $x_{n-1}$}\,.
  \end{split}
\end{equation}

\begin{lemma}[Standard variations of $U_T$]
  \label{lemma moments}
  If $n\ge 2$, $p\in(1,n)$, and $T>0$, then there are positive constants $R_0$ and $C_0$ depending on $n$, $p$, $T$, and $V_T$ such that the following properties hold:

  \medskip

  \noindent {\bf (i):} if $U\in\U_T$, then for every $|x|>R_0$ we have
  \begin{eqnarray}\label{UT decay pointwise}
  \begin{split}
    \hspace{1.1cm} \frac{1}{C_0|x|^{(n-p)/(p-1)}}\ &\leq U(x)\,\le \frac{C_0}{|x|^{(n-p)/(p-1)}}\,,\\
 \frac{1}{C_0|x|^{(n-1)/(p-1)}}& \leq |\nabla U(x)|\le \frac{C_0}{|x|^{(n-1)/(p-1)}}\,,
  \end{split}
  \end{eqnarray}
  and for every $R>R_0$,
  \begin{eqnarray}
  &&\!\!\!\!\!\!\!\!\!\!\!\!\label{UT decay pstar}
  \int_{H\setminus B_R}U^{\pstar}\le \frac{C_0}{R^{n/(p-1)}}\,,\qquad \int_{H\setminus B_R}U^{\pstar-1}\le \frac{C_0}{R^{p/(p-1)}}\,,
  \\
  &&\!\!\!\!\!\!\!\!\!\!\!\!\label{UT decay p}
  \int_{H\cap (B_{2R}\setminus B_R)}\!\!\!\!\!\!\!\!\!\!\!\!\!\!\!\!U^p\le \frac{C_0}{R^{(n-p^2)/(p-1)}}\,,\qquad
  \!\!\!\!\!\!\int_{H\setminus B_R}\!\!\!|\nabla U|^p\le \frac{C_0}{R^{(n-p)/(p-1)}}\,,
  \\
  &&\!\!\!\!\!\!\!\!\!\!\!\!\label{UT decay trace psharp}
  \int_{(\pa H)\setminus B_R}\!\!\!\!\!\!U^\psharp\le \frac{C_0}{R^{(n-1)/(p-1)}}\,,\qquad
  \int_{(\pa H)\setminus B_R}\!\!\!\!\!\!U^{\psharp-1}\le \frac{C_0}{R}\,,
  \\
  &&\!\!\!\!\!\!\!\!\!\!\!\!\label{UT decay pstar first moment}
  \int_{H\setminus B_R}|x|\,U^{\pstar}\le \frac{C_0}{R^{[1+n-p]/(p-1)}}\,,
  \\
  \label{UT decay grad p first moment}
  &&\!\!\!\!\!\!\!\!\!\!\!\!\int_{H\setminus B_R}|x|\,|\nabla U|^p\le \frac{C_0}{R^{(n+1-2p)/(p-1)}}\,,\qquad\mbox{if $n>2\,p-1$}\,.
  \end{eqnarray}

  \noindent {\bf (ii):} for every $U\in\U_T$ we have
  \begin{eqnarray}
  \label{UT variation gradient}
  &&\int_{H}|\nabla U|^p=\Phi_H(T)^p+p\,\l_H(T)\,t+{\rm o}(t)\,,
  \\
  \label{UT variation volume}
  &&\int_{H}U^\pstar=1+\pstar\,t+{\rm o}(t)\,,
  \\
  \label{UT variation trace}
  &&\int_{\pa H}U^\psharp=T^\psharp\,.
  \end{eqnarray}
\end{lemma}

\begin{proof}[Proof of Lemma \ref{lemma moments}] Since $V_T$ is assumed to be compactly supported, statement (i) follows immediately from the corresponding properties for $U_T$. More specifically, we deduce \eqref{UT decay pointwise} from \eqref{UT}, and \eqref{UT decay pstar}--\eqref{UT decay grad p first moment} from \eqref{UT decay pointwise}. Statement (ii) follows from \eqref{PDE of UT}.
\end{proof}

\subsection{The {\it Ansatz} for boundary concentrations} We next use the standard variations of minimizers of $\Phi_H$ described in Lemma \ref{lemma moments} to define certain competitors for $\Phi_\Om$ that provide us with a notion of ``standard boundary concentration.'' Recall the notation $U^{(\e)}$ for dilations introduced in \eqref{def of tau x0 and of rho alfa}.

\begin{lemma}\label{lemma boundary concentration}
 Fix $n\ge 2$, $p\in(1,n)$, $T>0$, $U \in \U_T$. Let  $\Om$ be an open set with $C^1$-boundary, $x_0\in\pa\Om$,  and let $\hat{f}= \hat{f}_{x_0}$, $\hat{g}=\hat{f}^{-1}$, $f= f_{x_0}$ and $g=f^{-1}$ be determined as in Remark \ref{remark fx0} starting from $\Om$ and $x_0$. Then the following statements hold:

 \medskip

 \noindent {\bf (i):} If $v\in W^{1,p}(\Om)$, $\b\in(0,1)$, $r_1=\e^\b$, $r_2=2\,\e^\b$, and $\vphi_\e$ is a cut-off function between $B_{r_1}(x_0)$ and $B_{r_2}(x_0)$ with $|\na \vphi_\e| \leq  C\e^{-\beta}$, then
 \begin{eqnarray}\label{def of veps hat}
 v_\e(x)=(1-\vphi_\e(x))\,v(x)+\vphi_\e(x)\,\big(U^{(\e)}\circ \hat{g}\big)(x)\,,\qquad x\in\Om\,,
 \end{eqnarray}
satisfies
 \begin{eqnarray}\label{veps converge grad}
 &&\lim_{\e\to 0^+}\int_{\Om}\!\!\!|\nabla v_\e|^p=\int_H\,|\nabla U|^p+\int_\Om|\nabla v|^p\,,
 \\\label{veps converge volume}
 &&\lim_{\e\to 0^+}\int_{\Om}\!\!\! v_\e^\pstar=\int_H U^\pstar+\int_\Om v^\pstar\,,
 \\\label{veps converge trace}
 &&\lim_{\e\to 0^+} \int_{\pa\Om} \!\!\!v_\e^\psharp=\int_{\pa H}U^\psharp+\int_{\pa\Om}v^\psharp\,.
 \end{eqnarray}

 \noindent {\bf (ii):} If $n>2p$, $v\in{\rm Lip}(\Om)$ and $\Om$ has $C^2$-boundary, then there exists a choice of $\b=\b(n,p)\in(0,1)$ (used in the definition of $r_1=\e^\b$ and $r_2=2\,\e^\b$), such that the function
  \begin{eqnarray}\label{def of veps}
 v_\e(x)=(1-\vphi_\e(x))\,v(x)+\vphi_\e(x)\,\big(U^{(\e)}\circ g \big)(x)\,,\qquad x\in\Om\,,
 \end{eqnarray}
 satisfies \eqref{veps converge grad}, \eqref{veps converge volume}, and \eqref{veps converge trace} in the more precise form
 \begin{align}
\label{veps stima grad}
\int_{\Om}|\nabla v_\e|^p & =\int_{H}|\nabla U|^p+\int_\Om|\nabla v|^p-{\rm H}_{\pa \Om}(x_0)\,\L(U)\,\e+{\rm o}(\e)\,,
\\
\int_{\Om}v_\e^\pstar  & = \int_H U^\pstar+\int_\Om v^\pstar-{\rm H}_{\pa \Om}(x_0)\,\M(U)\,\e+{\rm o}(\e)\,,
\label{veps stima vol}
\\\label{veps stima trace}
\int_{\pa \Om}v_\e^\psharp & =\int_{\pa H}U^{\psharp}+\int_{\pa\Om}v^\psharp+{\rm o}(\e)\,,
 \end{align}
as $\e\to0^+$. Here $\L(U)$ and $\M(U)$ are defined in \eqref{def of LU} and \eqref{def of Mu} and the orders in \eqref{veps stima grad}, \eqref{veps stima vol}, and \eqref{veps stima trace} depend on $n$, $p$, $T$, and $v$.
\end{lemma}

\begin{proof} Without loss of generality we assume that $x_0=0\in\pa\Om$, $T_0\,(\pa\Om)=\{x_n=0\}$ and $\nu_{\Om}(0)=-e_n$.  We carry out the proof in several steps.

\medskip

\noindent {\it Step one}: We start by noticing the following estimates for the energy, volume and trace of $v_\e$ in transition region for the cut-off function $\vphi_\e$. The estimates in this step hold in identical form with the same proofs for $v_\e$ defined from $\hat{f}$ as in \eqref{def of veps hat} and for $v_\e$ defined from $f$ as in \eqref{def of veps}; we write the proof for \eqref{def of veps}. First, with $v\in W^{1,p}(\Om)$,
\begin{equation}
  \label{gluing error senza rate}
  \lim_{\e\to0}\max\Big\{\int_{\Om\cap(B_{r_2}\setminus B_{r_1})}|\nabla v_\e|^p,\int_{\Om\cap(B_{r_2}\setminus B_{r_1})}v_\e^\pstar,\int_{(\pa\Om)\cap(B_{r_2}\setminus B_{r_1})}v_\e^\psharp\Big\}=0\,,
\end{equation}
and, second, under the additional assumption that $v\in {\rm Lip}(\Om)$,
\begin{eqnarray}
  &&\label{gluing error nabla p}
  \int_{\Om\cap(B_{r_2}\setminus B_{r_1})}|\nabla v_\e|^p\le C\,\max\Big\{\e^{(1-\b)\,(n-p)/(p-1)},\e^{\b(n-p)}\Big\}\,,
  \\\label{gluing error pstar}
  &&\int_{\Om\cap(B_{r_2}\setminus B_{r_1})}v_\e^\pstar\le  C\,\max\Big\{\e^{(1-\b)\,n/(p-1)},\e^{\b\,n}\Big\}
  \\\label{gluing error psharp}
  &&\int_{(\pa\Om)\cap(B_{r_2}\setminus B_{r_1})}v_\e^\psharp\le C\,\max\Big\{\e^{(1-\b)\,(n-1)/(p-1)}\,,\e^{\b(n-1)}\Big\}\,.
\end{eqnarray}
Indeed, we have $\nabla v_\e=a_\e+b_\e$ for
\[
a_\e=\vphi_\e\,(\nabla g)^*[(\nabla U^{(\e)})\circ g]+(U^{(\e)}\circ g)\,\nabla \vphi_\e\,,\qquad
b_\e=(1-\vphi_\e)\,\nabla v-v\,\nabla\vphi_\e\,.
\]
By \eqref{f inclusions}, and thanks to $|\nabla g|,Jf\le 2$ on $\CC_{r_0}$, we find
\begin{eqnarray*}
  &&\int_{\Om\cap(B_{r_2}\setminus B_{r_1})}|a_\e|^p\le
  C\,\int_{\Om\cap (B_{r_2}\setminus B_{r_1})}\!\!\!|(\nabla g)^*[(\nabla U^{(\e)})\circ g]|^p+\frac{U^{(\e)}(g)^p}{\e^{\b\,p}}
  \\
  &\le&
  C\,\int_{H\cap (B_{Cr_2}\setminus B_{r_1/C})}\!\!\!|\nabla U^{(\e)}|^p+\frac{(U^{(\e)})^p}{\e^{\b\,p}}
  \\
  &=&C\,\int_{H\cap (B_{Cr_2/\e}\setminus B_{r_1/C\e})}\!\!\!|\nabla U|^p+\e^{np/\pstar}\,\frac{U^p}{\e^{\b\,p}}\,\e^n
\\
  &\le&\!\!\!C\,\Big\{\e^{(1-\b)(n-p)/(p-1)}+\frac{\e^p\,\e^{(\b-1)(p^2-n)/(p-1)}}{\e^{\b\,p}}\Big\}=C\,\e^{(1-\b)(n-p)/(p-1)}\,,
\end{eqnarray*}
where in the last inequality we have used \eqref{UT decay p}. Concerning $b_\e$, we notice that if we only know that $v\in W^{1,p}(\Om)$ then  by $v\in L^\pstar(\Om)$ and $\nabla v\in L^p(\Om)$ we find that
\[
\int_{\Om\cap(B_{r_2}\setminus B_{r_1})}\!\!\!\!\!\!\!\!\!\!\!|b_\e|^p\le\int_{\Om\cap B_{r_2}}\!\!\!\!\!\!\!\!\!\!|\nabla v|^p+\int_{\Om\cap B_{r_2}}\frac{|v|^p}{\e^{\b\,p}}\le
\int_{\Om\cap B_{r_2}}|\nabla v|^p+\Big(\int_{\Om\cap B_{r_2}}\!\!\!\!|v|^\pstar\Big)^{p/\pstar}
\]
where the latter quantity converges to $0$ at an non-quantified rate as $\e\to 0^+$ (as stated in \eqref{gluing error senza rate}); while, if $v\in {\rm Lip}(\Om)$, then
\begin{eqnarray*}
    &&\!\!\!\!\!\!\!\!\!\!\!\!\!\int_{\Om\cap(B_{r_2}\setminus B_{r_1})}\!\!\!\!\!\!\!\!\!|b_\e|^p\le C\int_{\Om\cap(B_{r_2}\setminus B_{r_1})}\!\!\!\!\!\!\!\!|v|^p\,|\nabla\vphi_\e|^p+|\nabla v|^p\le C\,r_2^n\,{\rm Lip}(\vphi_\e)^p\le C\e^{\b\,(n-p)}\,,
\end{eqnarray*}
and \eqref{gluing error nabla p} is proved. The other two limits in \eqref{gluing error senza rate} follow similarly (with non-quantified rates), while if $v\in{\rm Lip}(\Om)$, then \eqref{gluing error pstar} and \eqref{gluing error psharp} follow from \eqref{UT decay pstar}, \eqref{UT decay trace psharp}, and
\begin{align*}
  \int_{\Om\cap(B_{r_2}\setminus B_{r_1})}u_\e^\pstar
  & \le C\e^{\beta n} + C\,\int_{H\cap(B_{Cr_2/\e}\setminus B_{r_1/C\e})}U^\pstar\le C\e^{\beta n} +C\,\e^{(1-\b)n/(p-1)}
  \\
  \int_{(\pa\Om)\cap(B_{r_2}\setminus B_{r_1})}u_\e^\psharp
 & \le C\e^{\beta(n-1)}+ C\,\int_{(\pa H)\cap(B_{Cr_2/\e}\setminus B_{r_1/C\e})}U^\psharp
  \le  C\e^{\beta(n-1)} + C\,\e^{(1-\b)\frac{(n-1)}{(p-1)}}\,.
\end{align*}

\medskip

\noindent {\it Step two}: We prove statement (i). By \eqref{gluing error senza rate},
\[
\int_\Om|\nabla v_\e|^p=\int_{\Om\cap B_{r_1}}\!\!\!|(\nabla \hat{g})^*[(\nabla U^{(\e)})\circ \hat{g}]|^p+\int_{\Om\setminus B_{r_2}}|\nabla v|^p+{\rm o}(1)\,,
\]
and, similarly,
\begin{equation}
  \label{nabla v is ok}
  \int_{\Om}|\nabla v|^p\ge\int_{\Om\setminus B_{r_2}}|\nabla v|^p\ge\int_{\Om}|\nabla v|^p+{\rm o}(1)\,.
\end{equation}
Moreover, if we set $E_\e = g(B_{r_1})\subset H $ and $\tilde{E}_\e = E_\e / \e \subset H$, then keeping in mind \eqref{f inclusions}, \eqref{f C1 estimates}, \eqref{f C1 estimate for Jf}, and \eqref{UT decay p}, we have
\begin{align*}
 \int_{\Om\cap B_{r_1}}\!\!\!|(\nabla \hat{g})^*[(\nabla U^{(\e)})\circ \hat{g}]|^p
   =  & \int_{E_\e }\!\!|((\nabla \hat{g})\circ \hat{f})^*[\nabla U^{(\e)}]|^p\,J\hat{f}
  = (1+{\rm o}(1)) \int_{E_\e }\!\!|\nabla U^{(\e)}|^p \\
  =(1+{\rm o}(1)) & \Big\{\int_{H}\!\!|\nabla U|^p
  -\int_{H\setminus \tilde{E_\e}}|\nabla U|^p\Big\}
  = (1+{\rm o}(1))  \int_{H}\!\!|\nabla U|^p\,.
\end{align*}
This proves \eqref{veps converge grad}. Entirely analogous arguments prove \eqref{veps converge volume} and \eqref{veps converge trace}.

\medskip

\noindent {\it Step three}: We now start the proof of statement (ii); in particular, from now on, $\Om$ has $C^2$-boundary,  $n>2p$, and $v_\e$ is defined as in \eqref{def of veps}; moreover, for the sake of brevity, we set $h={\rm H}_{\pa\Om}(0)$. In this step, we discuss the choice of $\b=\b(n,p)\in(0,1)$, which is determined by the rates in \eqref{gluing error nabla p}, \eqref{gluing error pstar} and \eqref{gluing error psharp}, and by the fact that in \eqref{veps stima grad}, \eqref{veps stima vol} and \eqref{veps stima trace} we want errors of size ${\rm o}(\e)$: therefore, by
\begin{eqnarray*}
\b\,\min\Big\{n,n-1,n-p\Big\}>1\qquad &\mbox{iff} & \qquad \b>\frac1{n-p}\,,
\\
(1-\b)\,\min\Big\{\frac{n-p}{p-1},\frac{n-1}{p-1},\frac{n}{p-1}\Big\}>1\,\qquad &\mbox{iff}&\qquad \b<\frac{n+1-2\,p}{n-p}\,,
\end{eqnarray*}
we are led to choose
\begin{equation}
  \label{def of beta}
  \b\in\Big(\frac1{n-p},\min\Big\{1,\frac{n+1-2\,p}{n-p}\Big\}\Big)\,,
\end{equation}
(where the interval appearing in \eqref{def of beta} is non-empty thanks to $n>2p$). With this choice of $\beta$, we have
$
\min\{\b\,(n-p),(1-\b)\,\frac{n-p}{p-1}\}>1\,,
$
and thus deduce from \eqref{gluing error nabla p}, \eqref{gluing error pstar} and \eqref{gluing error psharp} that
\begin{equation}
  \label{gluing error little o}
  \max\Big\{\int_{\Om\cap(B_{r_2}\setminus B_{r_1})}|\nabla v_\e|^p,\ \int_{\Om\cap(B_{r_2}\setminus B_{r_1})}v_\e^\pstar, \ \int_{(\pa\Om)\cap(B_{r_2}\setminus B_{r_1})}v_\e^\psharp\Big\}={\rm o}(\e)\,.
\end{equation}

\medskip

\noindent {\it Step four}: We prove \eqref{veps stima grad}. We first notice that by \eqref{gluing error little o} and \eqref{nabla v is ok} we have
\begin{equation}
  \label{veps stima grad proof 1}
  \int_\Om|\nabla v_\e|^p=\int_{\Om\cap B_{r_1}}\!\!\!|(\nabla g)^*[(\nabla U^{(\e)})\circ g]|^p+\int_{\Om}|\nabla v|^p+{\rm o}(\e)\,.
\end{equation}
Now, by \eqref{f C2 estimate for nabla g} we have
\begin{eqnarray*}
&&\nabla(U^{(\e)}\circ g)\circ f=[(\nabla g)\circ f]^*\,\nabla U^{(\e)}
\\
&=&\nabla U^{(\e)}+\pa_n U^{(\e)}\,\nabla\ell-(\nabla\ell\cdot\nabla U^{(\e)})\,e_n
+x_n\,\sum_{i=1}^{n-1}\k_i\,\pa_i U^{(\e)}\,e_i+{\rm O}(|x|^2)\,|\nabla U^{(\e)}|\,,
\end{eqnarray*}
so that, recalling that $|\nabla\ell|={\rm O}(|x|)$,
\begin{eqnarray*}
\Big|\nabla(U^{(\e)}\circ g)\circ f\Big|^2&=&|\nabla U^{(\e)}|^2
+2\,\big((\nabla\ell\cdot\nabla U^{(\e)})\,e_n-\pa_n U^{(\e)}\,\nabla\ell\big)\cdot\nabla U^{(\e)}
\\
&&+2\,x_n\,\sum_{i=1}^{n-1}\k_i\,(\pa_i U^{(\e)})^2+{\rm O}(|x|^2)\,|\nabla U^{(\e)}|^2
\\
&=&|\nabla U^{(\e)}|^2+2\,x_n\,\sum_{i=1}^{n-1}\k_i\,(\pa_i U^{(\e)})^2+{\rm O}(|x|^2)\,|\nabla U^{(\e)}|^2\,.
\end{eqnarray*}
Now set $a=|\nabla U^{(\e)}|$ and $b=[2\,\sum_{i=1}^{n-1}\k_i\,(\pa_i U^{(\e)})^2]^{1/2}$, so that $0\le b\le C\,a$ for a constant depending on $|A_{\pa\Om}(x_0)|$.
 Since $z\mapsto (1+z)^{p/2}$ is smooth in a neighborhood of $z=0$, we see that if $|x|<1/C$ for a constant $C$ depending on $|A_{\pa\Om}(x_0)|$, then
\begin{eqnarray*}
&&\big(a^2+x_n\,b^2+{\rm O}(|x|^2)\,a^2\big)^{p/2}= a^p\,\big(1+(b/a)^2\,x_n+{\rm O}(|x|^2)\big)^{p/2}
\\
&&=a^p\,\Big(1+\frac{p}2(b/a)^2\,x_n+{\rm O}(|x|^2)\Big)=a^p+\frac{p}2\,a^{p-2}\,b^2\,x_n+{\rm O}(|x|^2)
\end{eqnarray*}
and thus
\begin{align*}
\Big|\nabla(U^{(\e)}\circ g)\circ f\Big|^p\,Jf & =|\nabla U^{(\e)}|^p
\Big(1+p\,x_n\,\sum_{i=1}^{n-1}\k_i\frac{(\pa_i U^{(\e)})^2}{|\nabla U^{(\e)}|^2}+{\rm O}(|x|^2)\Big)\Big(1-x_n\,h+{\rm O}(|x|^2)\Big)\\
& =
|\nabla U^{(\e)}|^p-x_n\,\Big(h-p\,\sum_{i=1}^{n-1}\k_i\,\frac{(\pa_i U^{(\e)})^2}{|\nabla U^{(\e)}|^2}\Big)\,|\nabla U^{(\e)}|^p
+{\rm O}(|x|^2)\,|\nabla U^{(\e)}|^p\,.	
\end{align*}
Then, by \eqref{f inclusions},
\begin{eqnarray}\nonumber
&&\int_{\Om\cap B_{r_1}}\!\!\!|(\nabla g)^*[(\nabla U^{(\e)})\circ g]|^p\le\int_{H\cap B_{C\,r_1}}\big|[(\nabla g)\circ f]^*\,(\nabla U^{(\e)})\big|^p\,Jf
\\\nonumber
& & \le
\int_{H\cap B_{C\,r_1}}|\nabla U^{(\e)}|^p-h\,\int_{H\cap B_{C\,r_1}}  \!\! \!\!  x_n\,|\nabla U^{(\e)}|^p
\\\nonumber
&&+p\,\sum_{i=1}^{n-1}\k_i\,\int_{H\cap B_{C\,r_1}}  \!\! \!\!  x_n\,(\pa_i U^{(\e)})^2\,|\nabla U^{(\e)}|^{p-2}+C\,\int_{H\cap B_{C\,r_1}}  \!\! \!\!  |x|^2\,|\nabla U^{(\e)}|^p
\\\label{sup1}
&&=
\int_{H\cap B_{Cr_1/\e}} \!\! \!\! |\nabla U|^p-h\,\e\,\int_{H\cap B_{Cr_1/\e}} \!\!\!\! x_n\,|\nabla U|^p
\\\nonumber
&&+p\,\e\,\sum_{i=1}^{n-1}\k_i\,\int_{H\cap B_{Cr_1/\e}}  \!\!\!\! x_n\,(\pa_i U)^2\,|\nabla U|^{p-2}+C\,\e^2\,\int_{H\cap B_{Cr_1/\e}}  \!\!\!\!  |x|^2\,|\nabla U|^p\,.
\end{eqnarray}
Now, by the reflection symmetries of $U$ with respect to $\{x_i=0\}$, $i=1,...,n-1$ (recall \eqref{symmetry of U in UTeps}), we have
\[
\int_{H\cap B_R} x_n\,(\pa_i U)^2\,|\nabla U|^{p-2}=\int_{H\cap B_R} x_n\,(\pa_1 U)^2\,|\nabla U|^{p-2}\,,\qquad\forall i=1,...,n-1\,,\forall R>0\,,
\]
and therefore
\begin{eqnarray*}
  \sum_{i=1}^{n-1}\k_i\,\int_{H\cap B_{Cr_1/\e}} x_n\,(\pa_i U)^2\,|\nabla U|^{p-2}
  =h\,\int_{H\cap B_{Cr_1/\e}} x_n\,\big(\pa_1 U)^2\,|\nabla U|^{p-2}\,.
\end{eqnarray*}
Setting
\begin{eqnarray}\label{def of LUBR}
  \L(U,B_R)=\int_{H\cap B_R}x_n\,|\nabla U|^p-p\,x_n\,(\pa_1 U)^2\,|\nabla U|^{p-2}\,,
\end{eqnarray}
we can thus rewrite \eqref{sup1} as
\[
\int_{\Om\cap B_{r_1}}\!\!\!\!\!|(\nabla g)^*[(\nabla U^{(\e)})\circ g]|^p\le\int_{H}|\nabla U|^p
-h\,\e\,\L(U,B_{C\,r_1/\e})+C\,\e^2\,\int_{H\cap B_{C\,r_1/\e}}\!\!\!\!\!\!|x|^2\,|\nabla U|^p\,.
\]
At the same time, $\e^2\,|x|^2\le C\,\e\,r_1\,|x|$ for any $x\in B_{C\,r_1/\e}$,  so
\begin{eqnarray*}
  &&\e^2\,\int_{H\cap B_{Cr_1/\e}}\,|x|^2\,|\nabla U|^p\le C\,\e\,r_1\,\int_{H\cap B_{Cr_1/\e}}\,|x|\,|\nabla U|^p\,
  \\
  &&
  \le C\,\e^{1+\beta}\,\int_{H}\,|x|\,|\nabla U_T|^p +
  C\,\e^{1+\beta}\,\int_{H}\,|x|\,|\nabla V_T|^p\, \leq C(n,p,T)\,\e^{1+\beta},
\end{eqnarray*}
where we have used the facts that $V_T$ is compactly supported and that $n>2p-1$ to guarantee the convergence of the integrals in the final line.
Hence,
\begin{align*}
\int_{\Om\cap B_{r_1}}\!\!\!\!\!|(\nabla g)^*[(\nabla U^{(\e)})\circ g]|^p
& =\int_{H}|\nabla U|^p
-h\,\L(U,B_{C\,r_1/\e})\,\e+{\rm o}(\e)\,\\
& =\int_{H}|\nabla U|^p
-h\,\L(U)\,\e+{\rm o}(\e)\,,
\end{align*}
where the ${\rm o}(\e)$ term depends on $n$, $p$, and $T$, and in the second line we have applied \eqref{UT decay grad p first moment}.
By \eqref{veps stima grad proof 1} we deduce \eqref{veps stima grad}.

\medskip

\noindent{\it Step five}: We prove \eqref{veps stima vol}. We first notice that by \eqref{gluing error little o}, $v\in{\rm Lip}(\Om)$, $r_2^n=\e^{\b\,n}={\rm o}(\e)$ and our choice of $\b$ we have
\begin{equation}
  \label{veps stima vol proof}
\int_\Om v_\e^\pstar=\int_{\Om\cap B_{r_1}}\!\!\!(U^{(\e)}\circ g)^\pstar+\int_\Om v^\pstar+{\rm o}(\e)\,.
\end{equation}
Let $E_\e  = g(B_{r_1} \cap \Omega) \subset H$ and $\tilde{E}_\e = E_\e/\e \subset E$ as in step two.
Then keeping in mind
\eqref{f inclusions} and \eqref{f C2 estimate for Jacobian},
\begin{align*}
\int_{\Om\cap B_{r_1}}& \!\!\!(U^{(\e)}\circ g)^\pstar= \int_{E_\e }(U^{(\e)})^\pstar
  -h\,\int_{E_\e }x_n\,(U^{(\e)})^\pstar + \,\int_{E_\e }{\rm O}(|x|^2)\,(U^{(\e)})^\pstar
  \\
  &= \int_{\tilde{E}_\e }U^{\pstar}-h\,\e\,\int_{\tilde{E}_\e }x_n\,U^{\pstar}
  + \,\int_{E_\e }{\rm O}(|x|^2)\,(U^{(\e)})^\pstar
  \\
&=\int_H U^\pstar-h\,\e \mathcal{M}(U)
  + \Big\{-\int_{H\setminus \tilde{E}_\e } \!\!\!\!U^{\pstar}  + h\e \int_{H \setminus \tilde{E}_\e } x_n U^{\pstar} + \,\int_{E_\e }{\rm O}(|x|^2)\,(U^{(\e)})^\pstar\Big\}\,.
  \end{align*}
By \eqref{UT decay pstar} and \eqref{UT decay pstar first moment}, along with our choice of $\beta$, we see that
\[
-\int_{H\setminus \tilde{E}_\e}U^{\pstar} = o(\e), \qquad    h\,\e\, \int_{H \setminus \tilde{E}_\e } x_n U^{\pstar} = o(\e)\,.
\]
Moreover, since $U = U_T + tV_T$ with $V_T$ compactly supported in $H$ and $|t|\le 1$, we have
\begin{align*}
\int_{E_\e }|x|^2\,(U^{(\e)})^\pstar & \leq C\,r_1\,\e\,\int_{\tilde{E}_\e }|x|\,U^{\pstar}
\leq \e^{1+\beta} \int_{H} |x|\,U^{\pstar}\\
&  \leq C\e^{1+\beta}  \int_{H} |x|\,U_T^{\pstar} + C \e^{1+\beta}\int_{H} |x| V^{\pstar}  = {\rm o}(\e),
\end{align*}
with ${\rm o}(\e)$ depending on $n$, $p$, and $T$. So, the entire term in brackets above can be written as ${\rm o}(\e)$.
Combining this estimate with \eqref{veps stima vol proof}, we deduce \eqref{veps stima vol}.

\medskip

\noindent {\it Step six}: We finally prove \eqref{veps stima trace}. Notice that, by \eqref{gluing error little o}, $v\in{\rm Lip}(\Om)$, $r_2^{n-1}=\e^{\b\,(n-1)}={\rm o}(\e)$ (by the choice of $\b$), we have
\begin{equation}
  \label{veps stima trace proof}
\int_{\pa\Om} v_\e^\psharp=\int_{(\pa\Om)\cap B_{r_1}}\!\!\!(U^{(\e)}\circ g)^\psharp+\int_{\pa\Om} v^\psharp+{\rm o}(\e)\,.
\end{equation}
Now, by $J^{\pa H}f\ge 1$, \eqref{UT decay trace psharp} and our choice of $\b$ we have
\begin{align*}
\int_{(\pa\Om)\cap B_{r_1}}\!\!\!\!(U^{(\e)}\circ g)^\psharp  \ge
  \int_{(\pa H)\cap B_{r_1/C}} \!\!\!\!(U^{(\e)})^{\psharp}
  =\int_{\pa H}U^{\psharp}-\int_{(\pa H)\setminus B_{r_1/C\e}}\!\!\!\!\!\! U^{\psharp}
\ge\int_{\pa H}U^{\psharp}+{\rm o}(\e)\,.
\end{align*}
At the same time, by $J^{\pa H} f\le 1+C\,|x|^2$, we have
\begin{eqnarray*}
&&\int_{(\pa\Om)\cap B_{r_1}}\!\!\!(U^{(\e)}\circ g)^\psharp\le \int_{(\pa H)\cap B_{C\,r_1}}(1+C\,|x|^2)\,(U^{(\e)})^{\psharp}
\\
&&\le\int_{\pa H}U^{\psharp}+C\,\e^2\,\int_{(\pa H)\cap B_{C\,r_1/\e}}|x|^2\,U^{\psharp}
\end{eqnarray*}
where
\begin{eqnarray*}
&&\e^2\,\int_{(\pa H)\cap B_{C\,r_1/\e}}|x|^2\,U^{\psharp}\le C\,\e^2\,\int_0^{C\,r_1/\e}\,\frac{r^2\,r^{n-2}\,dr}{(r^{(n-p)/(p-1)})^\psharp}
\\
&&\le C\,\e^2\,\big(r_1/\e\big)^{1+n-(n-1)\,[p/(p-1)]}
\\
&&\le C\,\e^2\,\big(r_1/\e\big)^{(2p-n-1)/(p-1)}\le C\,\e^2\,\e^{(1-\b)\,(n+1-2p)/(p-1)}\le C\,\e^2={\rm o}(\e)\,,
\end{eqnarray*}
thanks to $n>2p-1$. This completes the proof.
\end{proof}

\section{Existence of minimizers}\label{section existence generalized}  We first establish the existence of generalized minimizers. Recall that $\Phi_\Om^*(T)$ was defined in \eqref{phi star Omega T}.

\begin{theorem}
  \label{theorem main existence part i implies}
  Let $n\ge 2$, $p\in(1,n)$, and let $\Om$ be a bounded open set with $C^1$-boundary in $\R^n$. Then:

  \medskip

  \noindent {\bf (i):} for every $T>0$, $\Phi_\Om(T)=\Phi_\Om^*(T)$;

  \medskip

  \noindent {\bf (ii):} there is a minimizer $(u,\vv,\tt)$ of $\Phi_\Om^*(T)$;

  \medskip

  \noindent {\bf (iii):} if $(u,\vv,\tt)$ is a minimizer of $\Phi_\Om^*(T)$ with $\int_\Om u^\pstar>0$, then $\int_{\pa\Om}u^\psharp>0$,
  \begin{equation}\label{u is also min}
  u\Big/\|u\|_{L^\pstar(\Om)}\,\,\mbox{is a minimizer of}\,\,\, \Phi_\Om\Big(\|u\|_{L^\psharp(\pa\Om)}\Big/\|u\|_{L^\pstar(\Om)}\Big)\,,
  \end{equation}
  and there exists $\l,\s\in\R$ such that
  \begin{equation}
  \label{PDE of generic minimizer theorem}
  \left\{
  \begin{split}
    &-\Delta_p u=\l\,u^{\pstar-1}\,,\quad\hspace{1cm}\mbox{on $\Om$}\,,
    \\
    &|\nabla u|^{p-2}\,\frac{\pa u}{\pa\nu_\Om}=\s\,u^{\psharp-1}\,,\quad\mbox{on $\pa\Om$}\,,
  \end{split}\right .
  \end{equation}
  In particular, $u\in{\rm Lip}(\Om)$. If, in addition, $\vv>0$, then $\l$ and $\s$ are given by
  \begin{equation}
    \label{multipliers gen min equation}
      \l=\vv^{p-\pstar}\,\l_H(\tt/\vv)\,,\qquad   \s=\vv^{p-\psharp}\,\s_H(\tt/\vv)\,,
  \end{equation}
  and, in particular,
 \begin{equation}
    \label{exclusion}
    \frac{\tt}{\vv}\in(0,T_0)\cup(T_E,\infty)\,.
  \end{equation}

\end{theorem}

\begin{proof}
{\it Step one}: Since $(u,0,0)\in\Y_\Om(T)$ if $u\in\X_\Om(T)$, we have $\Phi_\Om^*(T)\le\Phi_\Om(T)$. To prove the converse inequality it is enough to show that for every $(u,\vv,\tt)\in\Y_\Om(T)$,
\begin{equation}
  \label{phistar equal phi proof}
  \mbox{$\exists\, u_j\in\X_\Om(T)$ s.t.}\,\,\,
\lim_{j\to\infty}\int_\Om|\nabla u_j|^p=\int_\Om|\nabla u|^p+\vv^p\,\Phi_H\Big(\frac{\tt}{\vv}\Big)^p\,.
\end{equation}
Looking back at the definition of $\Y_\Om(T)$ in the paragraph preceding the statement of Theorem~\ref{theorem main existence}, we can assume without loss of generality that $\vv>0$ and $\tt>0$.
 Moreover, given $(u,\vv,\tt)\in\Y_\Om(T)$ with $\vv$ and $\tt$ positive we can easily find $(u_j,\vv_j,\tt_j)\in\Y_\Om(T)$ with $\vv_j$, $\tt_j$, $\int_\Om u_j^\pstar$, and $\int_{\pa\Om}u_j^\psharp$ positive and such that $\E(u_j,\vv_j,\tt_j)\to \E(u,\vv,\tt)$. By a diagonal argument, it is thus sufficient proving \eqref{phistar equal phi proof} under the assumption that $\int_\Om u^\pstar$ and $\int_{\pa\Om}u^\psharp$ are positive.
  This said, we apply Lemma \ref{lemma boundary concentration}{\bf (i)} with
\[
v=\frac{u}{\vv}\,,\qquad U=U_{\tt/\vv}\,,
\]
to find functions $v_j$ with $v_j=v$ on $\Om\setminus B_{2\,\e_j}(x_0)$ for some $x_0\in\pa\Om$ and $\e_j\to 0^+$, and with
\begin{eqnarray*}
&&\int_\Om|\nabla v_j|^p=\frac1{\vv^p}\int_\Om|\nabla u|^p+\Phi_H(\tt/\vv)^p+G_j\,,
\\
&&\int_\Om v_j^\pstar=\frac1{\vv^\pstar}\int_\Om u^\pstar+1+V_j
\\
&& \int_{\pa\Om} v_j^\psharp=\frac1{\vv^\psharp}\int_{\pa\Om}u^\psharp+(\tt/\vv)^\psharp+T_j\,,
\end{eqnarray*}
where $G_j$, $V_j$, $T_j\to 0$ as $j\to\infty$ at a rate depending on $n$, $p$, $\Om$, $\tt/\vv$ and $u$ only. By Lemma \ref{lemma volume fix}, there exist $\eta$ and $C$ depending on $n$, $p$, $\Om$, $\tt/\vv$ and $u$, but independent from $j$, such that for any $(a_j,b_j)$ with $|a_j| + |b_j|<\eta,$ we have  functions $w_j$ such that
  \begin{equation}
    \label{vt fix trace and volume fixed j}
    \int_{\pa\Om} |w_j|^\psharp=a_j+\int_{\pa\Om}|v_j|^\psharp\,,\qquad  \int_\Om |w_j|^\pstar=b_j+\int_\Om|v_j|^\pstar\,,
  \end{equation}
  \begin{equation}
    \label{vt fix cost for the gradient j}
    \Big|\int_\Om|\nabla w_j|^p-\int_\Om|\nabla v_j|^p\Big|\le C\,\big(|a_j|+|b_j|\big)\,.
  \end{equation}
  For $j$ large enough we can apply this statement with $a_j=-T_j$ and $b_j=-V_j$ to find a sequence $\{w_j\}_j$ with
   \begin{equation}
    \label{vt fix trace and volume fixed jj}
    \int_{\pa\Om} |w_j|^\psharp=\frac1{\vv^\psharp}\int_{\pa\Om}u^\psharp+(\tt/\vv)^\psharp=\frac{T^\psharp}{\vv^\psharp}\,,\qquad  \int_\Om |w_j|^\pstar=\frac1{\vv^\pstar}\int_\Om u^\pstar+1=\frac1{\vv^\pstar}\,,
  \end{equation}
  \begin{equation}
    \label{vt fix cost for the gradient jj}
    \Big|\int_\Om|\nabla w_j|^p-\frac1{\vv^p}\int_\Om|\nabla u|^p-\Phi_H(\tt/\vv)^p-G_j\Big|\le C\,\big(|T_j|+|V_j|\big)\,.
  \end{equation}
  Setting $u_j=\vv\,w_j$, we obtain a sequence in $\X_\Om(T)$ that satisfies \eqref{phistar equal phi proof}.

  \medskip

  \noindent {\it Step two}: We prove that there is a minimizer for the generalized problem $\Phi_\Om^*(T)$. By the argument in step one we can find a sequence $\{u_j\}_j$ in $\X_\Om(T)$ such that $\int_\Om|\nabla u_j|^p\to\Phi_\Om^*(T)^p$. By Lemma \ref{lemma lions}, the measures $\mu_j$, $\nu_j$ and $\tau_j$ defined in \eqref{def of muj nuj tauj} have subsequential weak-star limits $\mu$, $\nu$ and $\tau$ satisfying \eqref{nu}, \eqref{tau} and \eqref{mu} and \eqref{mu lower bound}.
  In particular, there is an at most countable set $\{x_i\}_{i\in I}\subset\ov\Om$ and corresponding $\vv_i>0$ and $\tt_i\ge0$ for every $i\in I$,
such that
  \begin{equation}
    \label{from cc 1}
      \Phi_\Om^*(T)^p=\lim_{j\to\infty}\int_\Om|\nabla u_j|^p\ge\int_\Om|\nabla u|^p+ S^p\,\sum_{i\in I\setminus I_{{\rm bd}}}\vv_i^p+\sum_{i\in I_{\rm bd}}\vv_i^p\,\Phi_H(\tt_i/\vv_i)^p\,,
  \end{equation}
  where  $u$ is the  subsequential weak limit of $u_j$, and
  \begin{equation}
    \label{from cc 2}
      1=\int_\Om u^\pstar+\sum_{i\in I}\,\vv_i^\pstar\,,\qquad  T^\psharp=\int_{\pa\Om} u^\psharp+\sum_{i\in I_{\rm bd}}\,\tt_i^\psharp\,.
  \end{equation}
  Now set
  \[
  \vv_c^\pstar=\sum_{i\in I}\,\vv_i^\pstar\,,\qquad \tt_c^\psharp=\sum_{i\in I_{\rm bd}}\,\tt_i^\psharp\,.
  \]
  By an immediate adaptation of the  proof of Lemma \ref{lemma boundary concentration} we can easily construct a sequence $\{W_j\}_j$ in $\X_H(\tt_c/\vv_c)$ with the property that
  \[
  \int_H|\nabla W_j|^p\to \sum_{i\in I}\Big(\frac{\vv_i}{\vv_c}\Big)^p\,\Phi_H(\tt_i/\vv_i)^p\,.
  \]

  Since $W_j\in\X_H(\tt_c/\vv_c)$ implies $\int_H|\nabla W_j|^p\ge\Phi_H(\tt_c/\vv_c)^p$, we deduce from \eqref{from cc 1} that
  \[
  \Phi_\Om^*(T)^p\ge\int_\Om|\nabla u|^p+ \vv_c^p\,\Phi_H(\tt_c/\vv_c)^p\,,
  \]
  while \eqref{from cc 2} gives $(u,\vv_c,\tt_c)\in\Y_\Om(T)$. This proves that $(u,\vv_c,\tt_c)$ is a minimizer of $\Phi_\Om^*(T)$.

  \medskip

  \noindent {\it Step three}: We finally prove statement (iii). If $(u,\vv,\tt)$ is a minimizer of $\Phi_\Om^*(T)$ with $\int_\Om u^\pstar>0$, it is immediate to deduce \eqref{u is also min}, and since $\Phi_\Om(0)$ does not admit minimizers, it must also be $\int_{\pa\Om}u^\psharp>0$. By Lemma \ref{lemma what about minimizers}, the Euler-Lagrange equation \eqref{PDE of generic minimizer theorem} for $u$ holds for some $\l,\s\in\R$ and $u\in{\rm Lip}(\Om)$. Assuming now that $\vv>0$, we can prove \eqref{multipliers gen min equation} by noticing that, given $\vphi\in C^\infty_c(\R^n)$, if we define
  \[
  \a(\de)=\Big(1-\int_\Om(u+\de\,\vphi)^\pstar\Big)^{1/\pstar}-\vv\,,\qquad
  \b(\de)=\Big(T^\psharp-\int_{\pa\Om}(u+\de\,\vphi)^\psharp\Big)^{1/\psharp}-\tt\,,
  \]
  then there is $\de_0>0$ such that $(u+\de\,\vphi,\vv+\a(\de),\tt+\b(\de))\in\Y_\Om(T)$ for every $|\de|<\de_0$. In particular,
  \begin{eqnarray*}
  0=\frac{d}{d\delta}\Big|_{\de=0}\,\int_\Om|\nabla (u+\de\vphi)|^p+(\vv+\a(\de))^p\,\Phi_H\Big(\frac{\tt+\b(\de)}{\vv+\a(\de)}\Big)^p\,,
  \end{eqnarray*}
  and exploiting \eqref{PDE of generic minimizer theorem} as well as
  \[
  \a'(0)=-\vv^{1-\pstar}\,\int_\Om u^{\pstar-1}\vphi\,,\qquad  \b'(0)=-\tt^{1-\psharp}\,\int_{\pa\Om} u^{\psharp-1}\vphi\,,
  \]
  and \eqref{lambdaT sigmaT basic identities} (i.e. $\Phi_H(T)=\l_H(T)+\s_H(T)\,T^\psharp$ and $\s_H(T)\,T^{\psharp-1}=\Phi_H(T)^{p-1}\Phi_H'(T)$ for every $T>0$), we see that
  \begin{eqnarray*}
    0&=&\int_\Om|\nabla u|^{p-2}\nabla u\cdot\nabla\vphi+    \vv^{p-1}\,\Phi_H(\tt/\vv)^{p-1}\,\Phi_H'(\tt/\vv)\,\b'(0)
    \\
    &&\hspace{3.4cm}+\Big\{\vv^{p-1}\,\Phi_H(\tt/\vv)^p-\tt\,\vv^{p-2}\,\Phi_H(\tt/\vv)^{p-1}\,\Phi_H'(\tt/\vv)\Big\}\a'(0)
    \\
    &=&\l\,\int_\Om u^{\pstar-1}\,\vphi+\s\,\int_{\pa\Om}u^{\psharp-1}\vphi-\vv^{p-1}\,\tt^{1-\psharp}\,\Phi_H(\tt/\vv)^{p-1}\,\Phi_H'(\tt/\vv)\,\int_{\pa\Om} u^{\psharp-1}\vphi
    \\
    &&-\vv^{p-\pstar}\,\Big\{\Phi_H(\tt/\vv)^p-(\tt/\vv)\,\Phi_H(\tt/\vv)^{p-1}\,\Phi_H'(\tt/\vv)\Big\}\int_\Om u^{\pstar-1}\vphi
    \\
    &=&\big(\s-\vv^{p-\psharp}\s_H(\tt/\vv)\big)\,\int_{\pa\Om} u^{\psharp-1}\vphi
    +\big(\l-\vv^{p-\pstar}\l_H(\tt/\vv)\big)\,\int_{\Om} u^{\pstar-1}\vphi\,.
    \end{eqnarray*}
  Testing with $\vphi\in C^\infty_c(\Om)$ we find $\l=\vv^{p-\pstar}\l_H(\tt/\vv)$, and testing with $\vphi=1$ on $\ov{\Om}$ then gives $\s=\vv^{p-\psharp}\s_H(\tt/\vv)$. We finally prove \eqref{exclusion}, i.e. $\tt/\vv\not\in[T_0,T_E]$. Indeed, combining the balance condition \eqref{balance condition} with \eqref{multipliers gen min equation} we find
  \begin{equation}
    \label{exclusion proof}
      \vv^{p-\pstar}\,\l_H(\tt/\vv)\,\int_\Om u^{\pstar-1}+\vv^{p-\psharp}\,\s_H(\tt/\vv)\,\int_{\pa\Om} u^{\psharp-1}=0\,,
  \end{equation}
  where by \eqref{sigmaT behavior} and \eqref{lambdaT behavior} we have $\l_H>0$ in $[T_0,T_E)$ and $\s_H>0$ on $(T_0,T_E]$ (with $\l_H(T_E)=\s_H(T_0)=0$ by continuity). If $\tt/\vv\in [T_0,T_E)$, then $\eqref{exclusion proof}$ implies $\int_\Om u^\pstar=0$, a contradiction; if $\tt/\vv=T_E$, then \eqref{exclusion proof} gives $u=0$ on $\pa\Om$, so that, by \eqref{u is also min}, $u$ is a minimizer of $\Phi_\Om(0)$, and thus $u$ is optimal in the Sobolev inequality on $\R^n$, so that $\Om=\R^n$, contradicting the fact that $\Om$ is bounded.
\end{proof}

We are now ready to prove Theorem \ref{theorem main existence}.

\begin{proof}
  [Proof of Theorem \ref{theorem main existence}] Statement (i) is an immediate consequence of Theorem \ref{theorem main existence part i implies}. We thus focus on statement (ii), and assume that $\Om$ is of class $C^2$
  and that  $n>2p$.  We want to prove that if $(u,\vv,\tt)$ is a minimizer of $\Phi_\Om^*(T)$, then $\vv=\tt=0$. We assume by way of contradiction that either $\vv>0$ {\it or} $\tt>0$; recalling the definition of  $\Y_\Om(T)$, this implies that $\vv>0$ {\it and} $\tt>0$. We apply Lemma \ref{lemma boundary concentration} {\bf (ii)} with the choice $(v,T)=(u/\vv,\tau)$ at a point $x_0\in\pa \Om$ of positive mean curvature, noting that  if $\vv=1$ then $v \equiv 0$ and that $v$ is Lipschitz continuous if $\vv \in (0,1)$ thanks to \eqref{u is also min} and Lemma \ref{lemma what about minimizers}. Then, for every $U\in\U_{\tau}$, we have
   \begin{eqnarray}\label{veps stima grad j}
&&\!\!\!\!\!\!\!\!\!\!\!\!\!\!\!\!\!\!\!\!\!\!\int_{\Om}|\nabla v_\e|^p\le\int_{H}|\nabla U|^p+\int_\Om|\nabla v|^p-{\rm H}_{\pa \Om}(0)\,\L(U)\,\e+{\rm o}(\e)\,,
\\
&&\!\!\!\!\!\!\!\!\!\!\!\!\!\!\!\!\!\!\!\!\!\!\!\int_{\Om}v_\e^\pstar=\int_H U^\pstar+\int_\Om v^\pstar
\!-{\rm H}_{\pa \Om}(0)\,\M(U)\,\e+{\rm o}(\e)\,,
\label{veps stima vol j}
\\\label{veps stima trace j}
&&\!\!\!\!\!\!\!\!\!\!\!\!\!\!\!\!\!\!\!\!\!\!\!\int_{\pa \Om}v_\e^\psharp=\int_{\pa H}U^{\psharp}+\int_{\pa\Om}v^\psharp+{\rm o}(\e)\,,
\end{eqnarray}
where $\L(U)$ and $\M(U)$ are defined in \eqref{def of LUBR} and \eqref{def of Mu}. We apply this with $U\in\U_{\tau}$ given by
\[
U=U_{\tau}+b\,\e\,V_{\tau}\,,\qquad |\e|<\frac{1}{|b|}\,,\qquad b=\frac{{\rm H}_{\pa \Om}(0)\,\M(U_\tau)}{\pstar}\,,
\]
The reason for the choice of $b$ will become apparent in a moment. Indeed, thanks to \eqref{UT variation gradient}, \eqref{UT variation volume} and \eqref{UT variation trace}, we have
\begin{eqnarray*}
  &&\int_{H}|\nabla U|^p=\Phi_H(\tau)^p+p\,b\,\l_H(\tau)\e+{\rm o}\big(\e\big)\,,
  \\
  &&\int_{H}U^\pstar=1+\pstar\,b\,\e+{\rm o}\big(\e\big)\,,
  \qquad \int_{\pa H}U^\psharp=\tau^\psharp\,,
  \end{eqnarray*}
which, combined with \eqref{veps stima grad j}, \eqref{veps stima vol j}, \eqref{veps stima trace j} and
\[
\frac{\Phi_\Om(T)^p}{\vv^p}=\Phi_H(\tau)^p+\int_\Om|\nabla v|^p,\qquad \frac1{\vv^\pstar}=1+\int_\Om\,v^\pstar\,,\qquad (T/\vv)^\psharp=\tau^\psharp+\int_{\pa\Om}v^\psharp\,,
\]
implies that $w_\e=\vv\,v_\e$ satisfies
\begin{eqnarray}\label{veps stima grad j end}
&&\int_{\Om}|\nabla w_\e|^p\le\Phi_\Om(T)^p
+ \Big\{p\,b\,\l_H(\tau)-{\rm H}_{\pa \Om}(0)\,\L(U)\Big\}\,\vv^p\e+{\rm o}(\e)\,,
\\
&&\int_{\Om}w_\e^\pstar=1
\!+\big\{\pstar\,b-{\rm H}_{\pa \Om}(0)\,\M(U)\big\}\,\vv^{\pstar}\e+{\rm o}(\e)=1+{\rm o}(\e)
\label{veps stima vol j end}
\\\label{veps stima trace j end}
&&\int_{\pa \Om}w_\e^\psharp=T^\psharp+{\rm o}(\e)\,,
\end{eqnarray}
where in \eqref{veps stima grad j end} we have used the choice of $b$ to deduce
\begin{eqnarray*}
\M(U)=\M(U_\tau)+\pstar\int_Hx_n\,U_\tau^{\pstar-1}\,V_\tau\,b\,\e+{\rm o}(\e)\,,\quad\big\{\pstar\,b-{\rm H}_{\pa \Om}(0)\,\M(U)\big\}\,\e={\rm o}(\e)\,.
\end{eqnarray*}
In the same spirit, by
\[
\lim_{\e\to 0^+}|\L(U_\tau+b\,\e\,V_\tau)-\L(U_\tau)|=0\,,
\]
we deduce from \eqref{veps stima grad j end} that
\begin{equation}\label{veps stima grad j end 2}
\begin{split}
	\int_{\Om}|\nabla w_\e|^p & \le\Phi_\Om(T)^p
-\Big\{\L(U_\tau)-\frac{(n-p)}n\,\M(U_\tau)\,\l_H(\tau)\Big\}\,{\rm H}_{\pa \Om}(0)\,\vv^p\e+{\rm o}(\e)\\
& 	\leq \Phi_\Om(T)^p
-C(n,p, \tau) \,{\rm H}_{\pa \Om}(0)\,\vv^p\e+{\rm o}(\e)\,,
\end{split}
\end{equation}
where in the second line we apply Lemma~\ref{lemma key}. We thus conclude that
\begin{eqnarray}\label{veps stima grad jj}
&&\!\!\!\!\!\!\!\!\!\!\!\!\!\!\!\!\!\!\!\!\!\!\int_{\Om}|\nabla w_\e|^p\le \int_\Om|\nabla u|^p+\vv^p\,\Phi_H(\tt/\vv)^p
-M\,\vv^p\,\e+{\rm o}(\e)\,.
\end{eqnarray}
It remains to modify the functions $w_\e$ to obtain $w_\e^* \in \X_\Om(T)$ also satisfying \eqref{veps stima grad jj}, allowing us to conclude the proof of the theorem by choosing $\e$ sufficiently small.  We will distinguish between two cases, applying Lemma~\ref{lemma volume fix} in different ways in the two cases.

\medskip

\noindent{\it Case one}: Suppose first that $\vv <1$ and thus  $\int_\Om u^{\pstar}>0$. This also implies that $\int_{\pa\Om}u^\psharp>0$ by Theorem \ref{theorem main existence part i implies}-(iii).  Taking into account \eqref{veps stima vol j end} and \eqref{veps stima trace j end}, we can thus apply Lemma \ref{lemma volume fix} in an analogous way to step one of the proof of Theorem \ref{theorem main existence part i implies} in order to slightly modify $w_\e$ into $w_\e^*\in\X_\Om(T)$ with
\begin{eqnarray*}
&&\Phi_\Om(T)^p\le \int_{\Om}|\nabla w_\e^*|^p =  \int_{\Om}|\nabla w_\e|^p+\,(\vv^\psharp+\vv^\pstar)\,{\rm o}(\e)
\\
&&\le \int_\Om|\nabla u|^p+\vv^p\,\Phi_H(\tt/\vv)^p-\frac{M}2\,\vv^p \, \e  =\Phi_\Om(T)^p-\frac{M}2\,\vv^p\,\e<\Phi_\Om(T)^p\,,
\end{eqnarray*}
thus reaching a contradiction.

\medskip

\noindent{\it Case two}: Next, suppose that $\vv=1$.  So, $\Phi_\Omega(T) = \Phi_H(T)$,  $u=v \equiv 0$ and $v_\e = \vphi_\e \,  (U^{(\e)}\circ g).$
In this case, we will pull the relevant quantities back to the half space $H$ and apply Lemma~\ref{lemma volume fix} there to correct the volume and trace constraints. More specifically, for $\e<\e_0$, the support of $\vphi_\e$ is entirely contained in the domain of the diffeomorphism $f_{x_0}$, and so we can define $\Psi_\e : H\to \R$ by $\Psi_\e(y) = \vphi_\e \circ f(\e \,y)$. In this way,  we can rewrite
\[
w_\e = v_\e = (\Psi_\e\, U)^{(\e)} \circ g.
\]
Thanks to \eqref{f inclusions}, $\Psi_\e$ is identically equal to one in $B_{\e^{\beta-1}/C}\cap H$ and vanishes outside of  $B_{C\,\e^{\beta-1}}\cap H$. Using the area formula, we rewrite \eqref{veps stima vol j end}, \eqref{veps stima trace j end}, and \eqref{veps stima grad jj} as
 \begin{align*}
 \int_H (\Psi_\e \, U  )^{\pstar}\, m_\e   & = \int_\Omega w_\e^{\pstar} dx  = 1+ o(\e )\,,\\
  \int_{\pa H} (\Psi_\e\, U)^{p^\sharp}\,\hat {m}_\e & =  \int_{\pa \Omega} w_{\e}^{p^\sharp}  = T^{p^\sharp} + o(\e ).
 \end{align*}
 \begin{align*}
\int_H | A_\e\, [\na (\Psi_\e\, U)]|^p\,m_\e   =  \int_\Omega |\na w_\e|^p   & \leq \Phi_H(T) -M \e + o(\e),
\end{align*}
where we have set
\[
m_\e(x) = Jf(\eps \, x)\,,\qquad \hat{m}_\e(x)= J^{\pa H} f(\eps\,x)\,,\qquad A_\e(x) = (\na g \circ f(\e \, x))^*\,.
\]
We now repeat the argument used in the proof of Lemma \ref{lemma volume fix}: exploiting the fact that
\begin{equation}
  \label{hps on veps H}
  \Psi_\e\, U=U_T \qquad\mbox{on $(H\cap B_{R})\setminus (\spt V_T)$}\,,\qquad R=\frac{\e_0^{\b-1}}{C}\,,
\end{equation}
as well as that both $\int_H (\Psi_\e \, U  )^{\pstar}\, m_\e $ and $\int_{\pa H} (\Psi_\e\, U)^{p^\sharp}\,\hat {m}_\e$ are positive and finite, we can easily find $\psi\in C^\infty_c((H\cap{B}_R)\setminus (\spt V_T))$ and $\vphi\in C^\infty_c((\R^n\cap{B}_R)\setminus (\spt V_T))$ such that
  \begin{eqnarray}
  \label{vt fix choice H}
    &&\int_{H}(\Psi_\e\,U_T)^{\pstar-1}\,m_\e\,\psi=\int_{\pa H}(\Psi_\e\,U_T)^{\psharp-1}\,\hat{m}_\e\vphi=1\,,
    \\
    \nonumber
    &&\int_{H}(\Psi_\e\,U_T)^{\pstar-1}\,m_\e\,\vphi=\int_{\pa H}(\Psi_\e\,U_T)^{\psharp-1}\,\hat{m}_\e\,\psi=0\,.
  \end{eqnarray}
  Correspondingly, we consider the maps $h_\e:\R^2\to\R^2$ by
  \[
  h_\e(s,t)\!=\!\Big(\int_{\pa H} \big(|v_\e+s\,\vphi+t\,\psi|^\psharp -|v_\e|^\psharp\big)\,\hat{m}_\e  ,\ \int_H \big(|v_\e+s\,\vphi+t\,\psi|^\pstar -|v_\e|^\pstar\big)\,m_\e \Big)\,.
  \]
  By \eqref{hps on veps H}, $h_\e\in C^{1,\a}(\R^2;\R^2)$ for some $\a=\a(n,p)\in(0,1)$, with
  \[
  \sup_{\e<\e_0}\|h_\e\|_{C^{1,\a}(\R^2;\R^2)}<\infty\,;
  \]
  moreover, $h_\e(0,0)=0$ and, by \eqref{vt fix choice H},
  \[
  \nabla h_\e(0,0)=\Big(\begin{array}{c c}
    \psharp & 0
    \\
    0 & \pstar
  \end{array}\Big)\,.
  \]
We can thus apply the inverse function theorem uniformly in $\e$, to obtain functions $W_\e^*: H\to \R$ with support in $B_{C\e^{\beta-1}}$ such that
 \begin{equation}
 	\label{eqn: all good on H}
 	\begin{split}
 	\int_H (W_{\eps}^*)^{\pstar} m_\e  =1\,, & \qquad
 	\int_{\pa H} (W_{\eps}^*)^{p^\sharp} \hat{m}_\e  = T^{p^\sharp}\,, \\
 	\bigg| \int_H | A_\e[ \na (\Psi_\e\, U)]|^p m_\e  \ &- \ \int_H | A_\e[ \na  W_\eps^*]|^p m_\e \bigg| = o(\eps).
 	\end{split}
 \end{equation}
 Finally, for $\e <\e_0$, define $w_\e^* : \Omega \to \R$ by $w_\e ^* = (W_\e^*)^{(\e)} \circ g$.  Changing variables once again, \eqref{eqn: all good on H} tells us that $w_\e \in \X_\Omega(T)$ and that
 \begin{align*}
\Phi_\Omega(T)^p  & \leq   \int_\Omega |\na w_\e^*|^p  = \int_\Omega | \na w_\e |^p + o(\e)\\
 & \leq \Phi_H(T)^p   -M \e +{\rm o}(\e)  \leq \Phi_H(T)^p   -\frac{M \e }{2 } < \Phi_H(T) =\Phi_\Omega(T),
 \end{align*}
 giving us a contradiction in this case as well. This completes the proof of the theorem.
\end{proof}

\section{Rigidity theorems for best {S}obolev inequalities}\label{section rigidity sobolev} In this section, we prove Theorem~\ref{theorem main rigidity}.

\begin{proof}
  [Proof of Theorem \ref{theorem main rigidity}] Rigidity under assumption (ii) is immediate by combining Theorem \ref{theorem main existence}-(ii) with \eqref{mv criterion for rigidity} and \eqref{mn comparison for balls}. Let us now consider assumption (i), namely, there is $T_*>0$ such that
  \begin{equation}
    \label{rigidity hp}
    \Phi_\Om(T)=\Phi_B(T)\qquad\forall T\in(0,T_*)\,.
  \end{equation}
   Without loss of generality we can assume that $T_*<\ISO(B)^{1/\psharp}$. We argue by contradiction and  assume that  $\Om$ is not a ball.

  \medskip

  By Theorem \ref{theorem main existence part i implies}, for every $T>0$ there is $(u_T,\vv_T,\tt_T)$ a minimizer of $\Phi_\Om^*(T)$. The basic idea of the proof will be to show that the trace-to-volume ratio of $u_T$ must be, on one hand, uniformly positive and, on the other hand, tending to zero as $T\to 0$, giving us a contradiction. Since $\Omega$ is connected and is not a ball by assumption, the rigidity criterion \eqref{mv criterion for rigidity} together with \eqref{mn comparison for balls} and \eqref{rigidity hp} tell us that a classical minimizer for $\Phi_\Om(T)$ cannot exist for $T\in(0,T_*)$, and so we immediately deduce that $\vv_T<1$ for all such $T$. In other words, if we set
  \[
  \nu_T=\big(1-\vv_T^\pstar)^{1/\pstar}=\|u_T\|_{L^\pstar(\Om)}\,,\qquad\tau_T=(T^\psharp-\tt^\psharp)^{1/\psharp}=\|u\|_{L^\psharp(\pa\Om)}\,,
  \]
 then
  $  \nu_T>0$ for all $T\in(0,T_*)$.  So, Theorem \ref{theorem main existence part i implies}-(iii) implies that
  \begin{equation}
    \label{uT su nuT minimizer of phiom}
  \mbox{$u_T/\nu_T$ is a minimizer of $\Phi_\Om(\tau_T/\nu_T)$}\,.
  \end{equation}
    In particular, this means that
  \begin{equation}
  	\label{enq: lower bound on ratio}
  \frac{\tau_T}{\nu_T} \geq T_* \qquad \text{for all } T \in (0,T_*)\,,
  \end{equation}
since as we noted above, no minimizer of $\Phi_\Om(\tilde{T})$ can exist for $ \tilde{T} = \tau_T/\nu_T < T_*$.
Since $T\geq \tau_T$, \eqref{enq: lower bound on ratio} tells us that $T /\nu_T \geq T_* $; rearranging this inequality gives us the following lower bound on $\vv_T$:
  \begin{eqnarray}
    \label{lower bound on vT}
    \vv_T^\pstar\ge1-(T/T_*)^\pstar\,\qquad\forall T\in(0,T_*)\,.
  \end{eqnarray}
From this, an upper bound on the ratio $\tt_T/\vv_T$ follows immediately:
  \begin{eqnarray}
    \label{ttT on vvT goes to zero}
    \frac{\tt_T}{\vv_T}\le \big(1-(T/T_*)^\pstar\big)^{-1/\pstar}\,T\,\qquad\forall T\in(0,T_*)\,.
  \end{eqnarray}
 In particular $\tt_T/\vv_T\to 0$ as $T\to 0^+$.

  \medskip

On the other hand, we will now use the Euler-Lagrange equation for $u_T$ to show that
  \begin{equation}
    \label{stranger things 2}
   \lim_{T\to 0^+}\frac{\tau_T}{\nu_T}=0\,,
  \end{equation}
  which is a clear contradiction to \eqref{enq: lower bound on ratio}.
  Indeed, by Theorem \ref{theorem main existence part i implies}-(iii) and \eqref{uT su nuT minimizer of phiom}, $u_T$ satisfies the Euler-Lagrange equation
  \begin{equation}
  \label{PDE of generic minimizer theorem proof rigidity}
  \left\{
  \begin{split}
    &-\Delta_p u_T=\vv_T^{p-\pstar}\,\l_H(\tt_T/\vv_T)\,u_T^{\pstar-1}\,\quad\hspace{1cm}\mbox{on $\Om$}\,,
    \\
    &|\nabla u_T|^{p-2}\,\frac{\pa u}{\pa\nu_\Om}=\vv_T^{p-\psharp}\,\s_H(\tt_T/\vv_T)\,u_T^{\psharp-1}\, \quad\mbox{on $\pa\Om$}\,;
  \end{split}\right .
  \end{equation}
  see \eqref{PDE of generic minimizer theorem}, \eqref{multipliers gen min equation} and \eqref{PDE of UT} (for the definition of$\l_H(T)$ and $\s_H(T)$). Testing \eqref{PDE of generic minimizer theorem proof rigidity} with $u_T$, we find that
  \begin{eqnarray}\nonumber
  \int_\Om|\nabla u_T|^p&=&\vv_T^{p-\pstar}\,\l_H\Big(\frac{\tt_T}{\vv_T}\Big)\,\int_\Om u_T^\pstar+\vv_T^{p-\psharp}\,\s_H\Big(\frac{\tt_T}{\vv_T}\Big)\,\int_{\pa\Om}u_T^\psharp
  \\\label{sereni}
  \nonumber
  &=&\vv_T^{p-\pstar}\,\l_H\Big(\frac{\tt_T}{\vv_T}\Big)\,\nu_T^\pstar+\vv_T^{p-\psharp}\,\s_H\Big(\frac{\tt_T}{\vv_T}\Big)\,\tau_T^\psharp\,.
  \end{eqnarray}
  After rearranging and multiplying through by $\vv_T^{p^\sharp-p}\nu_T^{-p^\sharp}>0$,  we arrive at the inequality
    \begin{equation}
    \label{strange things}
    -\s_H\Big(\frac{\tt_T}{\vv_T}\Big)\,\Big(\frac{\tau_T}{\nu_T}\Big)^\psharp\le \l_H\Big(\frac{\tt_T}{\vv_T}\Big)\,\Big(\frac{\nu_T}{\vv_T}\Big)^{\pstar-\psharp}\,.
  \end{equation}
  By \eqref{lambdaT behavior} and the continuity of $T\mapsto\l_H(T)$, there are $C>0$ and $T_{**}>0$ such that $\l_H(T)\in(1/C,C)$ for every $T\in(0,T_{**})$. Moreover, thanks to \eqref{lower bound on vT}, we can ask that $\vv_T\ge 1/C$ for $T\in(0,T_{**})$. In particular, by \eqref{ttT on vvT goes to zero} and up to further increasing $C$, if $T<1/C$, then $\tt_T/\vv_T<T_{**}$ and thus \eqref{strange things}, $\nu_T\le 1$ and $\vv_T\ge 1/C$ give
   \begin{equation}
    \label{stranger things}
  -\s_H\Big(\frac{\tt_T}{\vv_T}\Big)\,\Big(\frac{\tau_T}{\nu_T}\Big)^\psharp\le C\,.
  \end{equation}
  By \eqref{sigmaT behavior} and the fact that $\tt_T/\vv_T\to 0$ as $T\to 0$, we see that $\s_H(T)\to -\infty$ as $T\to 0^+$, so that \eqref{stranger things} implies \eqref{stranger things 2}. We reach a contradiction to \eqref{enq: lower bound on ratio}, completing the proof.
\end{proof}

\appendix

\section{Proof of Lemma \ref{lemma lions}}\label{appendix lions} We will use the {\it Brezis--Lieb lemma} (if $(X,\mu)$ is a measure space, $q\ge 1$, and $\{f_j\}_j$ is bounded in $L^q(X)$, then
\begin{equation}
  \label{brezis lieb lemma}
  \int_X |f|^q\,d\mu = \lim_{j \to \infty} \int_X |f_j|^q\,d\mu - \int_X |f_j - f|^q\,d\mu\,,
\end{equation}
provided $f$ is a $\mu$-a.e. limit of $\{f_j\}_j$ on $X$), and the two Sobolev-type inequalities
\begin{eqnarray}
  \label{sob 1}
  \|u\|_{L^\pstar(\Om)}\le C_1\,\Big(\|\nabla u\|_{L^p(\Om)}+\|u\|_{L^p(\Om)}\Big)\,,
  \\
  \label{sob 2}
  \|u\|_{L^\psharp(\pa\Om)}\le C_2\,\Big(\|\nabla u\|_{L^p(\Om)}+\|u\|_{L^p(\Om)}\Big)\,,
\end{eqnarray}
which are valid, with constants $C_1$ and $C_2$ depending on $n$, $p$ and $\Om$ only, as soon as $\Om$ is bounded and has Lipschitz boundary.

\begin{proof}
  [Proof of Lemma \ref{lemma lions}] {\it Step one}: Since $\Om$ is a bounded open set with Lipschitz boundary, $u_j\weak u$ as distributions in $\Om$, and $\{\nabla u_j\}_j$ is bounded in $L^p(\Om)$, standard considerations show that $u\in W^{1,p}(\Om)$, $\{u_j\}_j$ is bounded in $L^\pstar(\Om)$ and in $L^\psharp(\pa\Om)$, and, up to extracting subsequences, $u_j\to u$ in $L^q(\Om)$ for every $q\in[1,\pstar)$, $u_j\to u$ in $L^r(\pa\Om)$ for every $r<\psharp$, $u_j\to u$ pointwise $\L^n$-a.e. on $\Om$ and $\H^{n-1}$-a.e. on $\pa\Om$, and that the sequences of Radon measures $\{\nu_j\}_j$, $\{\tau_j\}_j$ and $\{\mu_j\}_j$ defined in \eqref{def of muj nuj tauj} admits weak-$*$ limits $\nu$, $\tau$ and $\mu$, with $\spt\,\nu$ and $\spt\,\mu$ contained in $\ov{\Om}$, and $\spt\tau$ contained in $\pa\Om$.

  \medskip

  \noindent {\it Step two}: We let $\tilde{\mu}$ denote the weak-$*$ limit of $|\nabla(u_j-u)|^p\,\L^n\llcorner\Om$ (which exists up to possibly extracting a further subsequence), and claim that $\tilde\mu(\{x\})=\mu(\{x\})$ for every $x\in\R^n$. Indeed, by the elementary inequality
  \[
  \big||v+w|^p-|w|^p\big|\le\e\,|v|^p+C(p,\e)\,|w|^p\,,\qquad v,w\in\R^n\,,\e>0\,,
  \]
  given $x\in\R^n$ and $r>0$ we see that
  \begin{eqnarray*}
    \Big|\int_{B_r(x)}|\nabla u_j|^p-\int_{B_r(x)}|\nabla(u_j-u)|^p\Big|\le\e\,\int_{B_r(x)}|\nabla(u_j-u)|^p+C(p,\e)\,\int_{B_r(x)}|\nabla u|^p\,,
  \end{eqnarray*}
  so that letting first $j\to\infty$ (for $r>0$ such that $\mu(\pa B_r(x))=\tilde\mu(\pa B_r(x))=0$) and then $r\to 0^+$ (along a generic sequence of radii), we find indeed $|\mu(\{x\})-\tilde\mu(\{x\})|\le\e\,\tilde\mu(\{x\})$ for every $\e>0$.

  \medskip

  \noindent {\it Step three}: If now pick $\eta\in C^\infty_c(\R^n)$ and set $\tilde{\nu}=\nu-|u|^\pstar\,\L^n\llcorner\Om$ (notice that, by lower semicontinuity, $\tilde\nu$ is a Radon measure), then, by exploiting, in order, $\nu_j\weakstar\nu$, \eqref{brezis lieb lemma}, \eqref{sob 1}, and  $u_j\to u$ in $L^p(\Om)$, we find
\begin{eqnarray*}
  &&\int_{\R^n}\,|\eta|^\pstar\,d\tilde{\nu}\le\liminf_{j\to\infty}\int_\Om\,|\eta|^\pstar\,(|u_j|^\pstar-|u|^\pstar)\,
  =\lim_{j\to\infty}\int_\Om\,|\eta\,(u_j-u)|^\pstar
  \\
  &&\le C_1^\pstar\,\lim_{j\to\infty}\Big(\|\nabla(\eta(u_j-u))\|_{L^p(\Om)}+\|\eta(u_j-u)\|_{L^p(\Om)}\Big)^{\pstar}
  \\
  &&=C_1^\pstar\,\lim_{j\to\infty}\,\|\nabla(\eta(u_j-u))\|_{L^p(\Om)}^\pstar=
  C_1^\pstar\,\lim_{j\to\infty}\,\|\eta\,\nabla(u_j-u))\|_{L^p(\Om)}^\pstar
  \\
  &&\le C_1^\pstar\,\Big(\int_{\R^n}|\eta|^p\,d\tilde{\mu}\Big)^{\pstar/p}\,.
\end{eqnarray*}
In particular, for every $x\in\R^n$ one has
\begin{equation}
  \label{tildenu tildemu}
  \tilde\nu(B_r(x))\le C_1^\pstar\,\tilde\mu(B_r(x))^{\pstar/p}\qquad\mbox{for a.e. $r>0$}\,.
\end{equation}
By \eqref{tildenu tildemu}, $\tilde\nu$ is absolutely continuous with respect to $\tilde\mu$, so that $\tilde\nu=f\,\tilde\mu$ where $f$ is such that, for $\tilde\mu$-a.e. $x\in\R^n$,
\[
f(x)=\lim_{r\to 0^+}\frac{\tilde\nu(B_r(x))}{\tilde\mu(B_r(x))}\,;
\]
in particular, again by \eqref{tildenu tildemu}, for $\tilde\mu$-a.e. $x\in\R^n$ we have
\[
f(x)\le C_1^\pstar\,\lim_{r\to 0^+}\tilde\mu(B_r(x))^{(\pstar/p)-1}=C_1^\pstar\,\mu(\{x\})^{(\pstar/p)-1}\,.
\]
In particular, as $\pstar>p$, if $X=\{x\in\R^n:\mu(\{x\})>0\}=\{x_i\}_{i\in I}\subset\ov\Om$ ($I$ at most countable) denotes the set of atoms of $\mu$, then $f(x)=0$ $\mu$-a.e. on $\R^n\setminus X$, and we have proved that
\begin{equation}
  \label{nu proof}
  \spt\tilde\nu\subset\{x_i\}_{i\in I}\,.
\end{equation}
An entirely analogous argument, this time based on \eqref{sob 2} rather than on \eqref{sob 1}, shows that, if $\tilde\tau=\tau-|u|^\psharp\,\H^{n-1}\llcorner\pa\Om$, then
\begin{equation}
  \label{tau proof}
\spt\tilde\tau\subset \{x_i\}_{i\in I}\cap\pa\Om\,.
\end{equation}
We have thus proved the validity of \eqref{nu}, \eqref{tau} and \eqref{mu} for suitable $\vv_i,\tt_i\ge0$ and $\gg_i>0$: and of course we can discard possible points $x_i$ with $\vv_i=0$ from these decompositions, and directly assume that $\vv_i>0$ for every $i$. The fact that $\gg_i\ge S\,\vv_i$ if $x_i\in\Om$ is immediate by repeating the above argument with arbitrary $\eta\in C^\infty_c(\Om)$ (in which case $C_1$ can be replaced by $1/S$). An analogous argument, this time using Lemma \ref{lemma boundary 0}, shows that $\gg_i\ge \vv_i\,\Phi_H(\tt_i/\vv_i)$ if $x_i\in\pa\Om$.
\end{proof}

\section{Proof of Lemmas \ref{lemma boundary 0} and \ref{lemma boundary}}\label{appendix boundary}
This section is dedicated to the proof of Lemmas~\ref{lemma boundary 0} and  ~\ref{lemma boundary}. We recall the standard Taylor expansions for the inverse and determinant of a matrix that is a perturbation of the identity:
\begin{align}\label{eqn: expand inverse}
(\text{Id}_{\R^n} +t A)^{-1} &= \text{Id}_{\R^n} - t A + O(t^2) \, ,\\
\label{eqn: expand det}
\det (\text{Id}_{\R^n} + tA)& = 1 + t\, \text{trace}A + O(t^2);
\end{align}
see for instance \cite[Lemma 17.4]{maggiBOOK}.

\begin{proof}[Proof of Lemma~\ref{lemma boundary 0}] Note that $\hat {f}$ can equivalently be written as $\hat{f}(x) = x + \ell({\bf p}(x)) e_n$.
We directly see that $\hat{f}$ maps $C^1$-diffeomorphically onto its image with inverse  $\hat{g}(y) = y - \ell({\bf p}(y))x_n$ and \eqref{f hat inclusions} holds because $\ell(0)=|\na \ell(0)|=0$ and $\ell$ is $C^1$.
We compute, in the standard basis for $\R^n$, that
\[
\na \hat{f}({\bf p}(x),x_n ) = \left( \begin{array}{ccc}
Id &  0 \\
\na \ell({\bf p}(x)) & 1 	
 \end{array}
\right)
= Id +  \left( \begin{array}{ccc}
0 &  0 \\
\na \ell({\bf p}(x)) & 0 	
 \end{array}
\right)\,.
\]
Since $\ell $ is $C^1$ with $\na \ell(0)=0$, this gives the first estimate in \eqref{f C1 estimates}. The second estimate in \eqref{f C1 estimates} follows in the same way using the explicit form of $g$ above.  The first estimate in \eqref{f C1 estimate for Jf} follows from \eqref{eqn: expand det} and the expression for $\na \hat{f}$ above. The second estimate in \eqref{f C1 estimate for Jf} follows because $\hat{f} =F$ on ${\bf D}_{r_0}$ (compare with \eqref{def of F}), and so
\[
J^{\pa H} \hat{f} = J^{\pa H} F = \sqrt{1 + |\na \ell|^2}\,.
\]
This completes the proof of the lemma.
\end{proof}

\begin{proof}
  [Proof of Lemma \ref{lemma boundary}] {\it Step one}: We compute some geometric quantities for $\pa \Omega$ using the graphical coordinates defined by the map $F$ given in \eqref{def of F}. We start with first order quantities. For $x \in {\bf D}_{r_0}$ and $i =1,\dots , n-1$, we set $\tau_i := dF_x(e_i)$, i.e.,
\begin{equation}
	\label{eqn: tangent vectors 1}
\tau_i =  \pa_i F(x) = e_i + \pa_i \ell (x) e_n,
\end{equation}
so that $\{ \tau_1, \dots , \tau_{n-1}\}$ forms a basis for $T_{F(x)} \pa \Omega$. Since $\ell$ is $C^2$ and $\na\ell (x) = 0$, we have
\begin{equation}\label{eqn: ON to 2nd order}
	g_{ij}=	\langle \tau_i, \tau_j\rangle_{\R^n}  = \delta_{ij}  + \pa_i \ell \, \pa _j\ell =
\delta_{ij}  + {\rm O}(|x|^2)\,.
\end{equation}
In other words, the metric coefficients in graphical coordinates are Euclidean up to second order in $|x|$.
The volume measure of $\pa \Om$ is given by
\[
J^{\pa H} F = \sqrt{\det g_{ij} }= \sqrt{1+ |\na \ell|^2}\, ;
\]
this immediately gives the second estimate in \eqref{f C2 estimate for Jacobian} since $f=F$ for $x \in {\bf D}_{r_0}$.
The inverse metric coefficients are
\begin{equation}\label{graphical coords: g upper ij}
g^{ij} = \delta^{ij} +\frac{\delta^{ia}\,\delta^{ib}\,\pa_a \ell \, \pa_b \ell}{ 1 +|\na \ell|^2 }\delta^{ij} + O(|x|^2)\,.	
\end{equation}
Recall that, without loss of generality, we have chosen an orthonormal basis for $\R^{n-1}\subset\R^n$ that diagonalizes the Hessian of $\ell$ at $x=0$. Let $\{ \kappa_1 ,\dots ,\kappa_{n-1}\}$ denote the eigenvalues.
 The second fundamental form of $\pa \Omega$ at $x$ is defined by $A_x(v,w) = - \langle d \nu_\Omega (v) , w\rangle$ for tangent vectors $v,w \in T_x(\pa \Omega)$.  Using the shorthand $\nu:= \nu_\Omega \circ F : {\bf D}_{r_0} \to S^{n-1}$, the coefficients of the second fundamental form in the coordinates defined by $F$ are given by $A_{ij} = \langle \pa_{ij} F , \nu\rangle_{\R^n}.$ Differentiating \eqref{eqn: tangent vectors 1} above, we have $
	\pa_{ij} F =\pa_j \tau_i = \pa_{ij}\ell\, e_n$, and thus for $i,j \in \{1,\dots, n-1\}$ we have
	\begin{equation}\label{graphical coords: A lower ij}
	A_{ij} = \langle \pa_{ij} F, \nu \rangle = \pa_{ij}\ell \,\langle e_n ,\nu\rangle = \frac{-\pa_{ij} \ell}{\sqrt{1+ |\na \ell|^2}} = -\pa_{ij}\ell + {\rm O}(|x|^2)\,.
		\end{equation}
We have
\begin{equation}\label{graphical coords: A upper lower}
\pa_i \nu = -A_i^j \tau_j\,,
\end{equation}
and  from \eqref{graphical coords: g upper ij} and \eqref{graphical coords: A lower ij} we directly compute that $A_i^j = g^{ik}A_{kj}$ is given by
\begin{align*}
	A_i^j = \frac{- \pa_{ij}\ell}{\sqrt{1+ |\na \ell|^2}}  +O(|x|^2) & = -\pa_{ij}\ell + O(|x|^2)\\
	& = -\pa_{ij}\ell(0) + O(|x|^2) = -\kappa_j\delta_{i}^j+ O(|x|^2).
\end{align*}
\medskip

\noindent{\it Step 2:} Next, we use the previous step to compute geometric quantities associated to the coordinates defined by $f$.
For $x \in {\bf C}_{r_0}$ and $i=1,\dots, n-1$, we note that $\pa_i {\bf p}(x) = e_i$ and thus from \eqref{graphical coords: A upper lower},
In what follows we will suppress the composition with ${\bf p}$ in our notation, writing for instance $\tau_i$ in place of $\tau_i \circ {\bf p}(x)$. For $i=1,\dots ,n-1$, we have
\begin{align}
		\label{graphical fermi tangent vectors}
	 \pa_if(x) &= \tau_i  - x_n \pa_i \nu({\bf p}(x))  \\
	\nonumber &= e_i + \pa_i \ell\, e_n + x_n A_i^j \tau_j	 = e_i + \pa_i \ell\, e_n -x_n \kappa_i e_i + O(|x|^2),\\
\text{ and }\qquad
	\label{graphical fermi nth tangent vector}
	 \pa_n f(x) &= -\nu = \frac{e_n-\na \ell }{\sqrt{1+ |\na \ell|^2}} 	 = e_n -\na \ell+ O(|x|^2).	
	\end{align}
Together \eqref{graphical fermi tangent vectors} and \eqref{graphical fermi nth tangent vector} can be expressed in consolidated form as
\begin{align*}
	\na f = \sum_{i=1}^n \pa_i f \otimes e_i
	& = \sum_{i=1}^n e_i \otimes e_i + e_n \otimes \na \ell -\na \ell \otimes e_n - x_n  \sum_{i,j =1}^{n-1} \kappa_i(0) e_i \otimes e_i + O(|x|^2)\\
	&= \text{Id}_{\R^n} + e_n \otimes \na \ell -\na \ell \otimes e_n - x_n  \sum_{i,j =1}^{n-1} \kappa_i (0) e_i \otimes e_i + O(|x|^2).
\end{align*}
In particular,  from \eqref{eqn: expand det} we see that the volume form is given by
\[
Jf= \sqrt{\det g_{ij}} = 1 - x_n H_{\Omega}(0) + O(|x|^2)\,,
\]
giving us the first estimate in \eqref{f C2 estimate for Jacobian}.
We see directly from the definition that $f$ is a $C^1$ map, and since we see from the expression for $\na f$ above that $\na f(0) = \text{Id}_{\R^n}$, we may apply the inverse function theorem to see that, up to decreasing $r_0$, $f$ defines a $C^1$ diffeomorphism onto its image. Letting $g= f^{-1}$ and using the expansion of the inverse \eqref{eqn: expand inverse}, we find
\begin{align*}
(\na g)\circ f&  =(\na f )^{-1} = \text{Id}_{\R^n} - e_n \otimes \na \ell+ \na \ell \otimes e_n  + x_n \sum_{i,j=1}^{n-1} \kappa_i(0) e_i \otimes e_i + O(|x|^2).	
\end{align*}
Finally, \eqref{f inclusions} follows from these expressions for $\na f$ and $\na g$, along with the assumptions that $\na \ell(0) = 0$. This completes the proof of the lemma.
\end{proof}

\bibliographystyle{alpha}
\bibliography{references}
\end{document}